\documentclass[a4paper,11pt,twoside,reqno]{amsart}
\usepackage{graphicx} 
\usepackage{amsopn,amstext,amsbsy,amsmath,amscd,amsthm,amsfonts,xfrac,xcolor,enumitem,comment}
\usepackage[margin=1in]{geometry}
\usepackage{slashed}
\usepackage{dsfont}
\usepackage[titletoc]{appendix}

\usepackage[utf8]{inputenc} 
\usepackage[all]{xy} 
\usepackage{enumitem}
\usepackage{hyperref}
\usepackage{url}
\usepackage{color}
\usepackage{mathtools}
\theoremstyle{plain}
\newtheorem{theorem}                 {\bf Theorem}      [section] 
\newtheorem{proposition}  [theorem]  {\bf Proposition}
\newtheorem{corollary}    [theorem]  {\bf Corollary}
\newtheorem{lemma}        [theorem]  {\bf Lemma}

\theoremstyle{definition}
\newtheorem{example}      [theorem]  {\bf Example}

\newtheorem{remark}       [theorem]  {\bf Remark}
\numberwithin{equation}{section}

\title{Uncoupled Dirac-Yang-Mills Pairs on Closed Riemannian Spin Manifolds}
\author{Adam Lindstr\"om}
\date{\today}

\def\restr#1#2{{
  \left.\kern-\nulldelimiterspace 
  #1 
  \vphantom{\big|} 
  \right|_{#2} 
  }}

\def\nab#1#2#3{\nabla^{\hbox{$\scriptstyle{#1}$}}_{\hbox{$\scriptstyle{#2}$}}{\hbox{$#3$}}}

\def\ip#1#2{\left< #1, #2 \right>}

\def\varip#1#2{\left( #1, #2 \right)}

\def\ind#1{\mathcal{#1}}

\def\snab#1#2#3{\hbox{$\nabla$\kern-.1em\raise 1.2 ex\hbox{$\scriptstyle{#1}$}\kern-.5em\lower 0.8 ex\hbox{$#2$}\kern-.0em{$#3$}}}

\def\nsnab#1{\hbox{$\nabla$\kern-.1em\raise 1.0 ex\hbox{$\scriptstyle{#1}$}}}

\newcommand{\D}{\slashed{D}}

\newcommand{\intprodl}{%
    \mathbin{\scalebox{1.5}{$\lrcorner$}}%
}

\renewcommand{\ind}{\mathrm{ind}}
\def\smooth#1{C^\infty(#1)}

\def\rn{\mathbb R}
\def\cn{\mathbb C}

\def\cp{\cn \text{\rm P}}
\def\s{\mathbb{S}}
\def\tr{\mathrm{Tr}}

\def\dtatzero{\restr{\frac{d}{dt}}{t=0}}
\def\Ad{\mathrm{Ad}}
\def\ad{\mathrm{ad}}
\def\End{\mathrm{End}}
\def\re{\mathfrak{Re}}
\def\im{\mathfrak{Im}}

\def\div{\mathrm{div}}
\def\asb#1#2#3{#1 \times_{#3} #2}
\def\floor#1{\left\lfloor#1\right\rfloor}
\def\md#1{k_{\mathrm{pm}}(#1)}
\def\ch{\mathrm{ch}}
\def\Alt{\mathrm{Alt}}
\newcommand{\diff}{\mathrm{d}}

\def \U#1{\mathrm{\bf U}(#1)}
\def \O#1{\mathrm{\bf O}(#1)}
\def \SU#1{\mathrm{\bf SU}(#1)}
\def \SO#1{\mathrm{\bf SO}(#1)}
\def \Sp#1{\mathrm{\bf Sp}(#1)}

\def \Spin#1{\mathrm{\bf Spin}(#1)}
\def \la#1{\mathfrak{#1}}
\def \g{\mathfrak{g}}

\def \u#1{\mathfrak{u}(#1)}

\def \su#1{\mathfrak{su}(#1)}

\def \sl#1#2{\mathfrak{sl}_{#1}(#2)}

\begin{document}

\begin{abstract}
We study the Dirac-Yang-Mills equations on closed spin manifolds  with a focus on uncoupled solutions, i.e. solutions for which the connection form satisfies the Yang-Mills equation. Such solutions require the Dirac current, a quadratic form on the spinor bundle, to vanish. We study the condition that this current vanishes on all harmonic spinors using perturbation theory and obtain a classification of the connection forms for which this holds, which we show contains an open and dense subset of connections. This has several implications for the generic dimension of the kernel of the Dirac operator. We further establish existence results for uncoupled solutions, in particular in dimension $4$ using the index theorem. Finally we generalize a construction method for twisted harmonic spinors to construct explicit uncoupled solutions on $4$-manifolds admitting twistor spinors and on spin manifolds of any dimension admitting parallel spinors.
\end{abstract}

\maketitle

\section{Introduction}

\subsection{Setting}
The Dirac-Yang-Mills functional constitutes a part of the Riemannian version of the full action of the standard model and models the interaction of a force field with a fermionic matter field, both of which are massless (see \cite{Ham} for details on the standard model). 
More precisely, let $(M,g)$ be a closed $m$-dimensional Riemannian spin manifold with complex spinor bundle $\Sigma M$. Further, let $G \to P \to M$ be a smooth principal $G$-bundle over $M$ with compact structure group $G$ and an $\Ad$-invariant inner product on its Lie algebra $\la{g}$. Let $\rho: G \to \U{V} \cong \U{N}$ be a unitary representation of $G$ and $E = \asb{P}{V}{\rho}$ the associated $\mathrm{rank} \ N$ Hermitian vector bundle on $M$. The Dirac-Yang-Mills functional $\mathcal{S}^{\mathrm{DYM}}: \mathcal{C}(P) \times \Gamma(\Sigma M \otimes E) $ is then

\begin{align}\label{eq: DYM-action}
    \mathcal{S}^{DYM}(\omega, \Psi) = \int_{M} |F_{\omega}|^2 + \ip{\D_{\omega} \Psi}{\Psi}\  \diff v_{g},
\end{align}

where $\omega \in \mathcal{C}(P)$ is a connection one-form on $G \to P \to M$ with curvature two-form $F_\omega \in \Omega^2(M;\ad(P))$, $\Psi$ is a section of the bundle of $E$-\emph{valued spinors} $\Sigma M \otimes E$ and $\D_{\omega}$ is the Dirac operator on $\Gamma(\Sigma M \otimes E)$ induced by $\omega$. Critical points $(\omega,\Psi)$ of \eqref{eq: DYM-action} will be referred to as \emph{Dirac-Yang-Mills} pairs and satisfy the \emph{Dirac-Yang-Mills equations} (see Proposition \ref{prop: EL-equations}):
\begin{align*}
    \delta_{\omega} F_{\omega} &= J(\psi),\\
    \D_{\omega}\psi &= 0.       
\end{align*}

Here $\delta_\omega$ is the formal $L_2$-adjoint of the exterior covariant derivative $d_\omega$ and the \emph{Dirac current} $J(\Psi)$ is an $\ad(P)$-valued one-form on $M$. In terms of orthonormal frames $\{e_k\}_{k=1}^m$ and $\{\sigma_\alpha\}_{\alpha=1}^{\dim \la{g}}$ for $TM$ and $\ad(P)$, respectively, it is locally given by 
\begin{align}\label{eq: Dirac current}
    J(\Psi) = -\frac{1}{2}\sum_{k = 1}^{m}\sum_{\alpha = 1}^{\dim \g} \ip{\Psi}{e_k\cdot \rho_*(\sigma_\alpha)\Psi}e^k\otimes\sigma_\alpha.
\end{align}
The Dirac-Yang-Mills functional can be written as a sum $\mathcal{S}^{\mathrm{DYM}}(\omega,\Psi) = \mathcal{S}^{\mathrm{YM}}(\omega) + \mathcal{S}^{\mathrm{Dirac}}(\omega,\Psi)$ of the \emph{Yang-Mills} functional:
\begin{align}\label{eq: YM-action}
    \mathcal{S}^{YM}(\omega) = \int_{M} |F_{\omega}|^2 \ \diff v_{g}
\end{align}
and the Dirac functional:
\begin{align*}
    \mathcal{S}^{\mathrm{Dirac}}(\omega,\Psi) = \int_M \ip{\D_\omega \Psi}{\Psi}\ \diff v_{g}.
\end{align*}
Critical points of \eqref{eq: YM-action} satisfy the \emph{Yang-Mills} equations
\begin{align}
    \delta_\omega F_\omega = 0
\end{align}
and are known as Yang-Mills connections.

By virtue of wick rotations, transforming between Minkowski space and Euclidean space, the pure Riemannian Yang-Mills theory has for a long time been a very active research topic in particle physics (see \cite{Sh} for a compilation of results).

The pure Yang-Mills theory has also been of great interest from a differential geometry perspective and in particular the moduli space of (anti-)self-dual solutions has played a key role in understanding the topology of $4$-manifolds \cite{Do2,DoKr}. The celebrated Atiyah-Singer index theorem as well as results such as the alternative proof of completeness of the ADHM construction found in \cite{GodCor} show a connection between Yang-Mills theory and the zero modes of the associated twisted Dirac operators.

From both a physics and a mathematics perspective it therefore seems natural to consider a model which couples a Riemannian Yang-Mills theory to a fermionic matter field, modelled by a spinor field which \eqref{eq: DYM-action} provides. 
In spite of this, the Dirac-Yang-Mills equations have received relatively little attention. Study of the system was initialized by Parker in \cite{Pa} which is mostly concerned with regularity and where, perhaps most notably, an extension of Uhlenbeck's theorem on removability of singularities to the Dirac-Yang-Mills system is obtained. The study of regularity is continued in \cite{Is,Li,Ot}.
Notably, there is as of yet very little in the prior research concerning existence of solutions, and it is the aim of this article to begin the work of filling this gap.

Due to it being unbounded both from above and below, establishing existence of critical points of \eqref{eq: DYM-action} is a challenging and interesting problem.
Those Dirac-Yang-Mills pairs $(\omega,\Psi)$ for which $\omega$ is also a Yang-Mills connection (i.e. $\delta_\omega F_\omega = 0 = J(\Psi)$) are referred to as \emph{uncoupled}, while solutions which are not uncoupled are known as \emph{coupled}. Uncoupled Dirac-Yang-Mills pairs can equivalently be characterised as those pairs $(\omega, \Psi)$ which are simultaneously critical to $\mathcal{S}^{\mathrm{YM}}$ and $\mathcal{S}^{\mathrm{Dirac}}$.

Taking inspiration from \cite{AmGi} where the similar problem of finding Dirac-harmonic maps is treated, the main focus of this article is obtain existence results for such uncoupled solutions from index theory.

A necessary ingredient is a vanishing result for the current, which we are able to obtain. By applying the tools of analytic perturbation theory we push this even further and give a characterisation of those connection forms for which the current vanishes on the kernel of the associated Dirac operator. We show this to be a generic property of a connection, by which we mean the set of connections for which it holds contains a dense open subset of $\mathcal{C}(P)$. From this we derive both a lower bound on the dimension of the space of uncoupled Dirac-Yang-Mills pairs and an obstruction result whenever one can show $\ker J = \{0\}$.

In the final part of the article we generalise a construction method for twisted harmonic spinors found in \cite{Hi1} and use it among other things to construct Dirac-Yang-Mills pairs on Calabi-Yau manifolds $(M,g)$ with Yang-Mills connection the Levi-Civita connection of $g$ and coefficient bundle the subbundle $\su{T^\cn M}$ of $\End(T^\cn M)$ consisting of traceless skew-Hermitian endomorphisms of the complexified tangent bundle.

We also dedicate a sizeable portion of the article to deriving the Dirac-Yang-Mills equations and stress-energy tensor in a coordinate free way. We do this to establish conventions for the remainder of the paper, to illustrate the invariances of the system and perhaps most importantly in order to facilitate further study of this system within a modern differential-geometric framework.

We further remark that, unless otherwise stated, all objects studied in this paper are assumed to be smooth. In particular, a Dirac-Yang-Mills pair is taken to mean a smooth critical point of the Dirac-Yang-Mills functional and hence a smooth solution, in the classical sense, to the Dirac-Yang-Mills equations.
\subsection{Main Results}
Our first concern is to better understand when the Dirac current vanishes. In a similar vein as in \cite{AmGi} we find a variational condition for this (see Theorem \ref{thm: variational criteria}) from which we derive results in several directions. The most immediate useful consequence is that the current always vanishes on chiral spinors in even dimensions:
\begin{theorem}\label{thm: int-vanishing on chiral}
    Let $(M,g)$ be an even-dimensional Riemannian spin manifold, $G\to P \to M$ a principal fibre bundle and $E = \asb{P}{V}{\rho}$ an associated Hermitian vector bundle. Then $J(\Psi) = 0$ for all $\Psi \in \Gamma(\Sigma ^\pm M \otimes E)$.
\end{theorem}
As the Dirac operator exchanges chiral subspaces this, combined with any existence result for twisted harmonic spinors and Yang-Mills connections, yields an existence result for uncoupled Dirac-Yang-Mills pairs. In this paper, the existence of twisted harmonic spinors will be supplied by the index theorem, yielding Dirac-Yang-Mills pairs for a large class of manifolds satisfying $\dim M \equiv 0 \ (\mathrm{mod} \ 4)$. 

We also ask the question, given an associated vector bundle $E\to M$, which connection one-forms $\omega \in \mathcal{C}(P)$ are \emph{decoupling}, i.e. satisfy $\ker\D_\omega \subset \ker J$. By applying the tools of analytic perturbation theory to the aforementioned variational condition of Theorem \ref{thm: variational criteria} we yield several fruitful results about such connections. The first is a characterisation in terms of the splitting off of eigenvalues from $0$:
\begin{theorem}\label{thm: int-decoupling classification}
A connection one-form $\omega \in \mathcal{C}(P)$ satisfies $\ker \D_\omega \subset \ker J$ if and only if for each $\eta \in \Omega^1(M;\ad(P))$ the following conditions are satisfied: 
\begin{enumerate}[label = (\roman*)]
    \item for $t \in (-\epsilon,\epsilon)$ the repeated eigenvalues of $\D_{\omega + t \eta}$ are analytic functions of $t$,
    \item any such eigenvalue $\lambda(t)$ of $\D_{\omega +t \eta}$, with $\lambda(0) = 0$ satisfies $\dtatzero\lambda(t) = 0$.
\end{enumerate}
\end{theorem}

We remark that the condition $(i)$ holds for any connection one-form (see Lemma \ref{lem: analytic prop of affine perturbations} below) $\omega \in \mathcal{C}(P)$ and so the content of the theorem is that condition $(ii)$ is equivalent to being decoupling. One class of connection one-forms satisfying condition $(ii)$ are those which are local minima of the map $\omega \mapsto \dim \ker(\D_\omega)$. These we refer to as \emph{perturbation minimal connections}.
\begin{proposition}[\ref{prop: perturbation minimal dense}]
    The set of perturbation minimal connections $\mathcal{C}_{\mathrm{pm}}(P)$ is a dense and open subset of $\mathcal{C}(P)$ with respect to the $C^\infty$ topology.
    Furthermore, the function $\omega \mapsto \dim \ker \D_\omega$ is constant on $\mathcal{C}_{\mathrm{pm}}(P)$.
\end{proposition}
Hence, the condition of being decoupling is generic on $\mathcal{C}(P)$. On the one hand, this yields a theoretical lower bound on the dimension of the space of uncoupled Dirac-Yang-Mills pairs in the form of:

\begin{theorem}\label{thm: int-perturbation minimal floor}
    Let $\omega \in \mathcal{C}(P)$ and let $\D_\omega$ be the associated Dirac operator acting on $E$-valued spinors. Suppose further that for some (and hence all) connection one-form $\vartheta \in \mathcal{C}_{\mathrm{pm}}(P)$ it holds that $k_{\mathrm{pm}}(E) := \dim \ker (\D_{\vartheta}) > 0$. Then there exists a subspace $V \subset \mathrm{ker}\D_\omega$ of dimension at least $k_{\mathrm{pm}}(E)$ such that $J(\psi)=0$ for all $\psi \in V$. In particular, for each Yang-Mills connection $\omega$ there is a $k_{\mathrm{pm}}(E)$-dimensional space of solutions $(\omega, \psi)$ to the Dirac-Yang-Mills equations with connection $\omega$.
\end{theorem}

On the other hand, the current can be used to obtain obstructions to the existence of harmonic twisted spinors for generic connections if one can show $\ker J = {0}$. As an example we show the following:
\begin{theorem}\label{thm: int-generic invertibility theorem}
Let $(M^3,g)$ be a closed $3$-dimensional Riemannian spin manifold endowed with an $G$-principal bundle $G\to P\to M$, where $G$ is either $\SU{N}$ or $\U{N}$ and let $E = \asb{P}{\cn^N}{\mathrm{st}}$. Then each decoupling connection one-form $\omega$ is perturbation minimal and the associated Dirac operator on $\Gamma(\Sigma M \otimes E)$ satisfies 
\begin{equation*}
    \dim \ker \D_\omega = 0. 
\end{equation*}
\end{theorem}

From the vanishing result \ref{thm: int-vanishing on chiral} and the Atiyah-Singer index theorem we obtain existence results for Dirac-Yang-Mills pairs depending on the dimension and $\hat{A}$-genus of $M$. In the case $\hat{A}[M] \neq 0$ we show:
\begin{theorem}\label{thm: int-A-hat nonzero}
    Let $(M^m,g)$ be a closed Riemannian spin manifold with $\hat{A}[M] \neq 0$ and let $G = \SU{N}$ and $G\to P \to M$ be a $G$-principal fibre bundle. Suppose further that one of the following holds:
    \begin{enumerate}[label = (\roman*)]
        \item $m = 4$,
        \item $m \equiv 0 \ (\mathrm{mod} \ 4)$ and $N = 2$.
    \end{enumerate}
    Then if $P$ admits Yang-Mills connections there exists an irreducible representation $\rho: G \to \U{V}$ such that there exists uncoupled Dirac-Yang-Mills pairs $(\omega,\Psi)$ on $M$ with $\Psi \in \Gamma(\Sigma M \otimes E =\asb{P}{V}{\rho})$ and $\omega \in \mathcal{C}(P)$.
\end{theorem}
If $\hat{A} = 0$ and $\dim M = 4$ we instead have the sufficient conditions:
\begin{theorem}\label{thm: int-A-hat = 0}
    Let $(M^4,g)$ be a $4$-dimensional Riemannian spin manifold with $\hat{A}[M] = 0$ and let $G \to P \to M$ be a non-trivial principal bundle with compact structure group and $\rho: G \to \U{V}$ a faithful  unitary representation, giving rise to a Hermitian vector bundle $E = \asb{P}{V}{\rho}$. If $P$ admits an (anti-)self-dual connection one-form $\vartheta$ then $k=\mathrm{ind}(\D^+_E) \neq 0$ and for any Yang-Mills connection there exist a $|k|$-dimensional space of uncoupled Dirac-Yang-Mills pairs $(\omega,\Psi)$.
\end{theorem}
One class of $4$-manifolds satisfying $\hat{A}[M] = 0$ are those admitting positive scalar curvature metrics. If we assume positive scalar curvature the lower bound given by the index is attained by each instanton connection showing them to be examples of perturbation minimal connections. Specialising to the case $M = \s^4$ we have
\begin{theorem}
    Let $g$ be a metric on $\s^4$ conformally equivalent to the standard round metric and let $G\to P \to \s^4$ be a principal bundle with structure group $G = \SU{2}$. Set $E = \ad(P)$. Then for each non-(anti-)self-dual Yang-Mills connection one-form $\omega \in \mathcal{C}(P)$ we have $\dim \ker(\D_\omega) > k_\mathrm{pm}(E)$.
\end{theorem}
The existence of such non-(anti-)self-dual connections is shown in \cite{MR1067574,MR1023811}. The proof uses an explicit construction of solutions of both chiralities using a method from \cite{Hi1} which then also yields Dirac-Yang-Mills pairs by Theorem \ref{thm: int-vanishing on chiral}. This method requires a Yang-Mills connection $\omega$ and a twistor spinor $\psi \in \Gamma(\Sigma M)$ and produces an $\ad(P)$-valued harmonic spinor $\Psi = F_\omega \cdot \psi$. Here "$\cdot$" is the Clifford tensor product to be described later. The Dirac-Yang-Mills pairs are obtained by projecting on chiral subspaces. As expressed in \cite{Hi1} the method requires $\dim M = 4$. In this article we generalise this construction and find that it can be employed in arbitrary dimension with the caveat that for $\dim M \neq 4$ we require $\psi$ to be parallel.

In particular, using this method we are able to construct uncoupled Dirac-Yang-Mills pairs on Calabi-Yau manifolds:
\begin{theorem}\label{thm: int-DYM on Calabi-Yaus}
    Let $(M,g)$ be a Calabi-Yau manifold and denote by $\omega_g$ and $F_g \in \Omega^2(M;\su{T^\cn M})$ the Levi-Civita connection one-form and curvature two-form of $g$, respectively. Let $\psi^+$ and $\psi^-$ be orthogonal parallel spinors of positive resp. negative chirality. Then $(\omega_g, F_g\cdot \psi^\pm)$ are Dirac-Yang-Mills pairs on $M$ with coefficient bundle $E = \su{T^\cn M}$
\end{theorem}

\subsection{Organization}
The first half of the paper is focused on presenting the Dirac-Yang-Mills equations and their basic properties in the global language of differential geometry. Section \ref{sec: E-L} is devoted to a derivation of the Dirac-Yang-Mills equations as the Euler Lagrange equations associated to \eqref{eq: DYM-action} and in Section \ref{sec: metric variations} we consider variations of the Dirac-Yang-Mills functional with respect to the metric and both confirm conformal invariance and the form of the Dirac-Yang-Mills stress-energy tensor. In Section \ref{sec: perturbation theory} we apply the tools of analytic perturbation theory to derive conditions under which the Dirac current vanishes. In Section \ref{sec: index theory} we make use of the Atiyah-Singer index theorem and Chern-Weil theory to give sufficient conditions on the data $((M,g), G\to P \to M, \rho, \omega)$ for the index of the twisted Dirac operator to be non-zero. Paired with the vanishing theorems of the previous Section this yields existence results for uncoupled solutions. The final Section \ref{sec: contructions} is devoted to a method for the construction of explicit solutions out of Yang-Mills connections.

\section{Conventions and notations}
\subsection{Differential forms}
Let $(M,g)$ be an $m$-dimensional Riemannian manifold and let $\{e_1,\dots, e_m\}$ be a local orthonormal frame for the tangent bundle $TM$ of $M$. Then we denote the dual frame of $T^*M$ by $\{e^1,\dots, e^m\}$, chosen such that $e^i(e_j) = g(e_i,e_j) = \delta^i_j$. We make use of the notation $\sharp: T^*M \to TM$ and $\flat: TM \to T^*M$ for the classical musical isomorphisms.

Wedge products of forms are defined inductively with the convention\footnote{
Some authors prefer to define wedge products without the binomial factor on the right-hand side.
}

\begin{equation}\label{eq-wedge-alt}
    \alpha \wedge \beta = \binom{k+\ell}{k} \, \Alt(\alpha \otimes \beta),
\end{equation}
for $\alpha \in \Omega^k(M), \beta \in \Omega^\ell(M)$, where the alternator map $\Alt : \otimes^k T^\ast M \to \otimes^k T^\ast M$ is given by 
\begin{equation*}
    \Alt(\xi)(X_1,\ldots,X_k) = \frac{1}{k!}\sum_{\sigma \in S_k} \text{sgn}(\sigma)\,\xi(X_{\sigma(1)},\ldots, X_{\sigma(k)}),
\end{equation*}
where $S_k$ is the symmetric group of $k$ elements and $\xi \in \Gamma(\otimes^k T^* M)$.
Let $E$ be a (real or complex) vector bundle over $M$. Sections of the bundle $E \otimes \Lambda^k M$ are referred to as $E$\emph{-valued }$k$\emph{-forms} and the set of smooth sections will be denoted by $\Omega^k(M;E)$. 
For a $k$-form $\Xi \in \Omega^k(M;E)$ taking values in some vector bundle $E$ over $M$ we denote its fully antisymmetric components with respect to a local orthonormal frame $\{e_1,\dots,e_m\}$ by $\Xi_{i_1,\dots,i_k} = \Xi(e_{i_1},\dots,e_{i_k})$. We then have
\begin{align*}
    \Xi &= \sum_{i_1,\dots,i_k = 1}^m \Xi_{i_1,\dots,i_k}\otimes e^{i_1}\otimes\cdots\otimes e^{i_k}\\
    &= \frac{1}{k!}\sum_{i_1,\dots,i_k = 1}^m \Xi_{i_1,\dots,i_k}\otimes e^{i_1}\wedge\cdots\wedge e^{i_k}\\
    &= \sum_{i_1 < \cdots i_k} \Xi_{i_1,\dots,i_k}\otimes e^{i_1}\wedge\cdots\wedge e^{i_k}.
\end{align*}
We define the bundle metric $\ip{\cdot}{\cdot}_g$ on $\Omega^k(M)$ as usual by requiring that $\{e^{i_1}\wedge \cdots \wedge e^{i_k} \ | \ 1\leq i_1 < i_k \leq m\}$ is an orthonormal frame for $\Omega^k(M)$. Explicitly, for $\alpha,\beta \in \Omega^{k}(M)$ we have
\begin{align*}
    \ip{\alpha}{\beta}_g = \frac{1}{k!}\sum_{i_1,\dots,i_k = 1}^m \alpha(e_{i_1},\dots,e_{i_k})\beta(e_{i_1},\dots,e_{i_k}).
\end{align*}

Note that, with this choice of inner product, the embedding $\Lambda^k M \hookrightarrow \otimes^k T^\ast M$ is an isometry provided that the inner product on the tensor bundle $\otimes^k T^\ast M$ is defined with an additional factor $\frac{1}{k!}$.
We define the insertion operation $\intprodl: TM \otimes \Lambda^k M \to \Lambda^{k-1} M$ by 
$$
    X\intprodl\alpha(X_1,\dots, X_{k-1}) = \alpha(X,X_1,\dots, X_{k-1})
$$ 
for $X \in \Gamma(TM)$ and $\alpha \in \Omega^k(M)$.
This induces an operation $\intprodl: T^*M \otimes \Lambda^k M \to  \Lambda^{k-1} M$ via $\theta \intprodl \alpha := \theta^\sharp \intprodl \alpha$.
If $(E,\ip{\cdot}{\cdot}_E)$ is a Hermitian vector bundle, we define a Hermitian bundle metric on $\Omega^k(M;E)$ by 
\begin{align*}
    \ip{\Xi}{\Theta}_g = \frac{1}{k!}\sum_{i_1,\dots,i_k}^m \ip{\Xi(e_{i_1},\dots e_{i_k})}{\Theta(e_{i_1},\dots e_{i_k})}_E
\end{align*}
for $\Xi,\Theta \in \Omega^k(M;E)$.

A metric-compatible connection $\nabla^E$ on $(E,\ip{\cdot}{\cdot}_E)$ induces a \emph{covariant exterior derivative} on $E$-valued $k$-forms by extending the formula
\begin{align}\label{eq: exterior covariant}
    d_E(\xi \otimes \alpha) = \xi \otimes d\alpha + (-1)^k \alpha \wedge \nabla^E \xi.
\end{align}

For a $k$-form $\alpha \in \Omega^k(M)$ we define its \emph{Hodge-dual} $\star\alpha \in \Omega^{m-k}(M)$ by the relation
\begin{align*}
    \star\alpha \wedge \beta = \ip{\alpha}{\beta}_g\diff v_g.
\end{align*}
We extend this to $E$-valued forms by letting the Hodge-operator $\star$ act only on the form part. This operator is in both cases an isomorphism with inverse $\star^{-1}:\Omega^k(M;E) \to \Omega^{m-k}(M;E)$ given by
\begin{align}\label{eq: Hodge inverse}
    \star^{-1}\Xi = (-1)^{k(m-k)}\star\Xi.
\end{align}
If $m = \dim(M) \equiv 0 \ (\textrm{mod}\ 4)$ then $\star$ is an involutive automorphism of $\Omega^{\frac{m}{2}}(M;E)$. The eigensections corresponding to the eigenvalues of $1$ and $-1$ are referred to as self-dual and anti-self-dual $\frac{m}{2}$-forms respectively.

By the \emph{covariant co-differential} induced by $\nabla^E$ we shall mean the formal $L^2$ adjoint\footnote{
\label{foot: alternative notation for cov co-diff}
Some authors prefer the notation $d^*_E$ for the $L^2$ adjoint of $d_E$.
} $\delta_E: \Omega^{k}(M;E) \to \Omega^{k-1}(M;E)$ of the exterior covariant derivative $d_E$. As with the ordinary co-differential, we can express $\delta_E$ in terms of the Hodge-operator as 
\begin{align}\label{eq: covariant co-diff}
    \delta_E = (-1)^{k}\star^{-1}d_E\star = (-1)^{m(k+1)+1}\star d_E \star.
\end{align}
We have the following well-known and convenient local formulae for $d_E$ and $\delta_E$.
\begin{proposition}
    Let $(M,g)$ be an $m$-dimensional Riemannian manifold and let $E\xrightarrow{\pi} M$ be a Hermitian vector bundle, equipped with a metric connection $\nabla^E$. Then for any twisted $k$-form $\Xi \in \Omega^k(M;E)$ we have
    \begin{align*}
        d_E \Xi = \sum_{i=1}^m e^i \wedge \nabla^E_{e_i} \Xi
    \end{align*}
    and 
    \begin{align*}
        \delta_E \Xi = -\sum_{i=1}^m e^i \intprodl \nabla^E_{e_i} \Xi = (-1)^{m(k+1)+1}\sum_{i=1}^m\star e^i \wedge \star \nabla^E_{e_i} \Xi,
    \end{align*}
    where $\{e_1,\dots,e_m\}$ is a local orthonormal frame for $TM$.
\end{proposition}

\subsection{Principal bundles}
Let us also briefly recall some elements of gauge theory.
For a more extensive review of the topic, we recommend the standard textbooks \cite{Bau,Bl}.
Let $G$ be a compact Lie group, equipped with an Ad-invariant positive-definite inner product on the Lie algebra $\g$, and let $P \xrightarrow\pi M$ be a principal $G$-bundle. We shall henceforth use the shorthand notation $G\to P \to M$ for a principal $G$ bundle over $M$ with total space $P$.
Let $\rho : G \to \mathbf{GL}(V)$ be a representation of $G$ on the vector space $V$ with associated vector bundle $E = P \times_\rho V \to M$.
A connection form $\omega \in \Omega^1(P,\mathfrak g)$ induces an affine connection $\nabla^\omega$ on $E$ given explicitly by
\begin{align}\label{eq: Local derivative}
    \nabla^\omega_X \xi = [\sigma, X(\underline{\xi}) + \rho_{*}(\sigma^*(\omega)(X))(\underline{\xi})], 
\end{align}
where $\sigma: U\subset M \to P$ is a local gauge and $\xi\in \Gamma(E)$ is locally represented by $\xi = [\sigma,\underline{\xi}]$. This also induces a twisted exterior derivative $d_\omega = d_{\nabla^\omega}$ on $\Omega^{k}(M;E)$ via \eqref{eq: exterior covariant}.

The difference $\eta = \omega_1 -\omega_2$ between two connection forms can be pulled down to yield a globally defined $\ad(P)$-valued form $\eta \in \Omega^1(M;\ad(P))$. 
Moreover, each $\ad(P)$-valued one-form $\eta$ occurs as the difference between two connection forms on $P$, so that the set $\mathcal{C}(P)$ of connection form on $P$ has the structure of an affine space over $\Omega^1(M;\ad(P))$. 

The curvature two-form $F_\omega \in \Omega^2(P;\g)$ of a connection form $\omega$ is given by
\begin{align}\label{eq: structure eqn curvature}
    F_\omega = d\omega + \frac{1}{2}[\omega,\omega],
\end{align}
where the bracket\footnote{
It is also common to encounter the notation $[\cdot \wedge\cdot]$ or simply $\cdot \wedge \cdot$ for this operation.
} $[\cdot,\cdot]: \Omega^1(P;\g)\times\Omega^1(P;\g) \to \Omega^2(P,\g)$ is given by
$$
[\eta,\theta](X,Y) = [\eta(X),\theta(Y)]-[\eta(Y),\theta(X)].
$$
The curvature form pulls down to yield an $\ad(P)$-valued two-form $F_\omega \in \Omega^2(M;\ad(P))$. From now on, unless otherwise explicitly stated, we shall consider the curvature form of a connection to be this $\ad(P)$-valued two-form on the base manifold. The curvature two-form relates to the curvature tensor $R^E_\omega \in \Omega^2(M; \End E)$ of the induced affine connection $\nabla^\omega$ on $E = P\times_{\rho} V$ via
\begin{align}\label{eq. curvature form tensor relation}
    R^E_\omega(X,Y) = \rho_*(F_\omega(X,Y))
\end{align}
and so we shall henceforth denote this tensor by $\rho_*(F_\omega)$. The curvature of a connection form also satisfies the Bianchi identity
\begin{align}\label{eq: Bianchi}
    d_\omega F_\omega = 0.
\end{align}
We say a connection form $\omega \in \mathcal{C}(P)$ on $P$ is a \emph{Yang-Mills connection} if it satisfies the Yang-Mills equation
\begin{align}\label{eq:YM-equation}
    \delta_\omega F_\omega = 0.
\end{align}
Yang-Mills connections are critical points of the Yang-Mills functional \eqref{eq: YM-action}. When $\dim M = 4$ we have a special class of Yang-Mills connections which we shall refer to as (anti-)self-dual, satisfying the first order equation (in $\omega$):
\begin{align}\label{eq: instanton eqn}
    \star F_\omega = \pm F_\omega,
\end{align}
and which correspond to global minima of the Yang-Mills functional.
\subsection{Spinors and the Dirac sector}
On a spin manifold $(M,g)$ with fixed spin structure $\chi$ and with spinor bundle $\Sigma M$ the Dirac operator is given by $\slashed{\partial} = \mathfrak{m}\circ \nabla^{\Sigma M}:\Gamma(\Sigma M) \to \Gamma(\Sigma M)$, where $\mathfrak{m}:\Gamma(T^*M \otimes \Sigma M)\to \Gamma(\Sigma M)$ is Clifford multiplication. Locally we have
\begin{equation*}
    \slashed{\partial} \psi = \sum_{k=1}^m e_k \cdot \nab{\Sigma M}{e_k}{\psi},
\end{equation*}
for a local orthonormal frame $\{e_k\}_{k=1}^m$ of $TM$.
If $E\to M$ is another smooth complex vector bundle over $M$ we shall refer to the tensor product bundle $\Sigma M \otimes E$ as the bundle of $E$-valued spinors, or as the bundle of spinors twisted by $E$. Equipping $E$ with a Hermitian metric $\ip{\cdot}{\cdot}$ and metric connection $\nabla^E$ yields, respectively, a Hermitian bundle metric $\ip{\cdot}{\cdot}$ and metric connection $\widetilde{\nabla}$ on $\Sigma M \otimes E$ via
$$
    \ip{\psi \otimes \xi}{\phi \otimes \eta}_{\Sigma M \otimes E} = \ip{\psi}{\phi}_{\Sigma M}\ip{\xi}{\eta}_E
$$
and
$$
\widetilde{\nabla}(\psi \otimes \xi) = \nabla^{\Sigma M}\psi \otimes \xi + \psi \otimes \nabla^{E}\xi
$$
for $\psi,\phi \in \Gamma(\Sigma M)$, $\xi,\eta \in \Gamma(E)$.

Clifford multiplication is extended to $\Sigma M \otimes E$ by standard Clifford multiplication on the first factor.
The connection $\nabla^E$ induces a \emph{twisted Dirac operator} $\D:\Gamma(\Sigma M \otimes E \to \Sigma M \otimes E)$ with the local expression
\begin{align*}
    \D\Psi = \sum_{k=1}^m e_k \cdot \widetilde{\nabla}_{e_k}\Psi
\end{align*}
for $\Psi\in \Gamma(\Sigma M \otimes E)$.

Both twisted and untwisted Dirac operators are first-order self-adjoint elliptic differential operators.

In this paper we shall make use of several notions of Clifford products between (twisted) $k$-forms and (twisted) spinors, which will all be denoted by "$\cdot$". However, it should always be clear from context which is meant.

Let $E$ be a Hermitian vector bundle over $M$. We define the \emph{Clifford tensor product} $\cdot: \Omega^k(M;E) \times \Gamma(\Sigma M) \to \Gamma(\Sigma M \otimes E)$ by
\begin{align}\label{eq: twisted form dot spinor}
    \Theta \cdot \psi = \sum_{i_1 < \cdots < i_k} \Theta^{i_1,\dots, i_k} \otimes e_{i_1}\cdots e_{i_k}\cdot  \psi,
\end{align}
for $\Theta \in \Omega^k(M;E)$ and $\psi \in \Gamma(\Sigma M)$.

We define the \emph{Clifford action} $\cdot: \Omega^k(M; \End(E))\times \Gamma(\Sigma M \otimes E) \to \Gamma(\Sigma M \otimes E)$ by
\begin{align}\label{eq: endo-form dot twisted}
    \Xi \cdot \Psi = \sum_{i_1 < \cdots < i_k} e_{i_1}\cdots e_{i_k}\cdot \Xi^{i_1,\dots, i_k}(\Psi),
\end{align}
for $\Xi \in \Omega^k(M; \End(E))$ and $\Psi \in \Gamma(\Sigma M \otimes E)$. 
Here, the Clifford multiplication is on the $\Sigma M$-factor while the components of $F$ act on the $E$-factor.
Both the Clifford tensor product and the Clifford action can be extended by linearity to forms of mixed degree.

With this notation, the well-known Weizenb\"ock formula for twisted Dirac operators (Theorem 8.17 on p.164  in \cite{LaMi}) is given as follows: 
\begin{proposition}\label{prop: WB}[\cite{LaMi}]
    Let $(M,g)$ be a Riemannian spin manifold with spinor bundle $\Sigma M$. Let $\pi:P\to M$ be a principal $G$-bundle, $E = \asb{P}{V}{\rho}$ an associated Hermitian vector bundle and $\omega \in \mathcal{C}(P)$ a connection one-form on $P$ inducing a connection $\nabla^\omega$ on $E$. Then the associated twisted Dirac operator $\D_\omega$ on $\Gamma(\Sigma M \otimes E)$ satisfies
    \begin{align}\label{eq: Weizenbock}
        \D_\omega^2 \Psi = \Delta_\omega \Psi + \frac{1}{4}S_g\Psi + \rho_*(F_\omega)\cdot \Psi
    \end{align}
    for all $\Psi \in \Gamma(\Sigma M \otimes E)$, where $\Delta_\omega = {\widetilde{\nabla}^{\omega\star}} \widetilde{\nabla}^\omega$ is the connection Laplacian on $\Gamma(\Sigma M \otimes E)$ and $S_g$ the scalar curvature of $(M,g)$.
\end{proposition}
\section{The Euler Lagrange Equations}\label{sec: E-L}

\begin{proposition}\label{prop: EL-equations}
    The Euler-Lagrange equations for $\mathcal{S}^{DYM}$ given by \eqref{eq: DYM-action} are
    \begin{align}\label{eq: DYM equations}
    \begin{split}
        \delta_{\omega} F_{\omega} &= J(\Psi),\\
        \D_{\omega}\Psi &= 0,
    \end{split}        
    \end{align}
    where $F_{\omega} \in \Omega^{2}(M; \ad(P))$ denotes the curvature two-form of the connection form $\omega$, $\delta_{\omega}$ is the formal $L^2$-adjoint of the exterior covariant derivative $d_{\omega}$ acting on $\ad(P)$-valued differential forms and $J(\Psi) \in \Omega^{1}(M; \ad(P))$ is the Dirac current given by \eqref{eq: Dirac current}.
    \begin{proof}
        We first compute the variation of $\mathcal{S}^{DYM}$ with respect to the connection, while keeping the spinor $\Psi\in \Gamma(\Sigma M \otimes E)$ fixed.
        Recall that the space $\mathcal{C}(P)$ of connection one-forms on $P$ is an affine space over the vector space $\Omega^1(M;\ad(P))$ of $\ad(P)$-valued one-forms on $M$ in the sense that for two connection forms $\omega_1,\omega_2$ their difference
        \begin{align*}
            \eta = \omega_1 -\omega_2 \in \Omega^{1}(P;\g)
        \end{align*}
        pulls down to a globally defined $\ad(P)$-valued one-form on $M$ which we shall also denote by $\eta$.
        
        Fix $\omega \in \mathcal{C}(P)$ and $\Psi \in \Gamma(\Sigma M \otimes E)$ and let $t \mapsto \omega_t$ be a smooth one-parameter variation with $\omega_0 = \omega$ and $\dtatzero\omega_t =: \eta \in \Omega^{1}(M;\ad(P))$. 
        
        We also set $\eta_t := \omega_t-\omega \in \Omega^1(M;\ad(P))$.
        The structure equations for the curvature form
        $$
        F_{\omega_t} = d\omega_t + \frac{1}{2}[\omega_t,\omega_t]
        $$
        yield
        \begin{align*}
            F_{\omega_t} &= d\eta_t + d\omega  + \frac{1}{2}\left([\eta_t,\eta_t] + [\eta_t,\omega] + [\omega,\eta_t] +[\omega,\omega]\right)\\
            &= F_\omega + d\eta_t + [\omega,\eta_t] +\frac{1}{2}[\eta_t,\eta_t],
        \end{align*} 
        where we are considering $F_{\omega}, \eta_t$ and $\omega$ as $\g$-valued forms on the total space $P$. 
        One checks that
        $$
            d_\omega \eta_t = d\eta_t + [\omega,\eta_t] 
        $$
        so that we have
        \begin{align*}
            F_{\omega_t} = F_{\omega} + d_\omega \eta_t + \frac{1}{2}[\eta_t,\eta_t].
        \end{align*}
        This holds both in the sense of $\g$-valued forms on $P$ and $\ad(P)$-valued forms on $M$. Hence
        \begin{align}\label{eq: derivative curvature}
            \dtatzero F_{\omega_t} = d_\omega \eta
        \end{align}
        with $\eta = \dtatzero \eta_t$.

        From the formula \eqref{eq: Local derivative} we see that
        \begin{align*}
            \nabla^{\omega_t} - \nabla^\omega = \rho_*(\eta_t),
        \end{align*}
        where $\rho_*(\eta_t) \in \Omega^1(M;\mathrm{End}(E))$ is given by
        \begin{align*}
            \rho_*(\eta_t)(X)\xi = [\sigma, \rho_*(\eta(X))\underline{\xi}]
        \end{align*}
        for $X \in \Gamma(TM)$, $\sigma: U \subset M \to P$ a local gauge and $\xi \in \Gamma(E)$ locally represented by $[\sigma,\underline{\xi}]$.
        
        For the induced covariant derivatives on $\Gamma(\Sigma M \otimes E)$
        we then get 
        \begin{align*}
            \widetilde{\nabla}^{\omega_t} - \widetilde{\nabla}^\omega = \mathrm{id}_{\Sigma M} \otimes \rho_*\eta_t.
        \end{align*}
        From now on we shall write $\rho_*\eta_t$ also for $\mathrm{id}_{\Sigma M} \otimes \rho_*\eta_t \in \Omega^1(M;\End(\Sigma M \otimes E))$.
        Then 
        \begin{align}\label{eq: difference of Diracs}
            \D_{\omega_t} - \D_{\omega} = \mathfrak{m}\circ (\widetilde{\nabla}^{\omega_t} - \widetilde{\nabla}^\omega) = \mathfrak{m} \circ \rho_*(\eta_t).
        \end{align}
        The operator occurring on the right in \eqref{eq: difference of Diracs} will appear frequently throughout the paper, so we define
        \begin{align}\label{eq: definition of K_eta}
            K_{\eta} := \mathfrak{m}\circ \rho_*(\eta) \in \End(\Sigma M \otimes E)
        \end{align}
        for $\eta \in \Omega^1(M;\ad(P))$. 
        We then have
        \begin{align}\label{eq: derivative of Dirac}
            \dtatzero\D_{\omega_t} = K_{\eta}.
        \end{align}
        From \eqref{eq: derivative curvature} and \eqref{eq: derivative of Dirac} we obtain
        \begin{align}\label{eq: variation wrt connection}
                \restr{\frac{d \mathcal{S}^{DYM}(\omega_t,\Psi)}{dt}}{t = 0} = \int_{M}2\ip{d_{\omega}\eta}{F_{\omega}} + \ip{\Psi}{K_{\eta}\Psi} \diff v_g.
        \end{align}
        At a critical point $(\omega,\Psi)$ we then have
        \begin{align}\label{eq: DYM 1 integral form}
            \int_{M} \ip{d_{\omega}\eta}{F_{\omega}} \diff v_{g} = -\frac{1}{2}\int_{M}\ip{\Psi}{K_{\eta}\Psi} \diff v_{g}, \quad \textrm{for all } \eta \in \Omega^1(M; \ad(P)).
        \end{align}
        To obtain a differential equation from the integral equation \eqref{eq: DYM 1 integral form} we introduce the \emph{Dirac current} $J(\Psi) \in \Omega^1(M;\ad(P))$ by requiring that
        \begin{align*}
            -\frac{1}{2}\ip{\Psi}{K_{\eta}\Psi}_{\Sigma M \otimes E} = \ip{\eta}{J(\Psi)}_{\Omega^1(M;\ad(P))}
        \end{align*}
        holds for all $\eta \in \Omega^1(M;\ad(P))$. In terms of local orthonormal frames $\{\sigma_\alpha\}_{\alpha =1}^{\dim \g}$ and $\{e_k\}_{k=1}^{m}$ for $\ad(P)$ and $TM$ respectively, one finds the expression
        \begin{align*}
            J(\Psi) = -\frac{1}{2}\sum_{k,\alpha = 1}\ip{\Psi}{e_k\cdot \rho_*(\sigma_\alpha)\Psi}e^k\otimes\sigma_\alpha
        \end{align*}
        for $J(\Psi)$. Then \eqref{eq: DYM 1 integral form} can be rewritten as
        \begin{align*}
            \int_{M}\ip{\eta}{\delta_{\omega}F_{\omega}} \diff v_g = \int_{M} \ip{\eta}{J(\Psi)} \diff v_g, \quad \textrm{for all } \eta \in \Omega^1(M;\ad(P))
        \end{align*}
        and hence at a critical point we have
        \begin{align}\label{eq: DYM 1}
            \delta_{\omega} F_{\omega} = J(\Psi).
        \end{align} 
        We now instead consider a smooth one-parameter variation $t \mapsto \Psi_t$ of $\Psi$ with $\dtatzero\Psi_t =: \phi$, this time keeping the connection form $\omega$ fixed. 
        By linearity and formal self-adjointedness of $\D_{\omega}$ we then have
        \begin{align*}
            \restr{\frac{d\mathcal{S}^{DYM}(\omega, \Psi_t)}{dt}}{t= 0} &=\int_M \ip{\Psi}{\D_\omega \phi} + \ip{\phi}{\D_\omega\Psi} \diff v_g\\
            &= 2\int_M \mathrm{Re} \ip{\phi}{\D_\omega \Psi} \diff v_g. 
        \end{align*}
        Hence at a critical point $(\omega,\Psi)$ we must have
        \begin{align}\label{eq: DYM 2}
            \D_\omega \Psi = 0
        \end{align}
        which concludes the proof.
        \end{proof}
\end{proposition}

\section{Variations with respect to the metric, Conformal invariance and the stress-energy tensor}\label{sec: metric variations}

In this section we confirm  the conformal invariance of the Dirac-Yang-Mills action in dimension $4$ and derive an expression for the stress-energy tensor in any dimension. We shall first give an overview of how the spinor bundles associated with different metrics on the same base manifold are related. For a detailed description we refer to \cite{BoGa} and \cite{KiFr}.

We shall denote by $\SO{M,g}$ and $\Spin{M,g}$ the oriented orthonormal frame bundle and spin bundle, respectively of the Riemannian spin manifold $(M,g)$. We shall also denote by $\Sigma M_g$ the spinor bundle $\Sigma M_g = \asb{\Spin{M,g}}{\Sigma}{\sigma}$.

Let $g$ and $h$ be two metrics on $M$ and denote by $h_g$ the symmetric positive definite $(1,1)$-tensor field given by $h(X,Y) = g(X,h_{g}Y)$ for all $X,Y \in \Gamma(TM)$. The tensor field $h_{g}^{-\frac{1}{2}}$ then gives rise to an $\SO{m}$-equivariant bundle isomorphism $b^{g,h}: \SO{M,g} \to \SO{M,h}$ via
$$
    b^{g,h}: (e_1,\dots, e_m) \mapsto (h_{g}^{-\frac{1}{2}}e_1,\dots,h_{g}^{-\frac{1}{2}}e_m).
$$
This lifts to a $\Spin{m}$-equivariant isomorphism $\widetilde{b}^{g,h}: \Spin{M,g} \to \Spin{M,h}$. We pay special attention to the case of $h$ being conformally equivalent to $g$, i.e. $h = e^{2u}g$ for some $u \in C^\infty(M)$. In this case $h_g = e^{2u}$ and the bundle isomorphism $b^{g,h}: \SO{M,g} \to \SO{M, h}$ is given by
$$
b^{g,h}(e_1,\dots, e_m) \mapsto (e^{-u}e_1, \dots, e^{-u}e_m).
$$
These principal bundle isomorphisms induce Euclidean resp. Hermitian vector bundle isometries $\beta^{g,h}: TM \to TM$ and $\widetilde{\beta}^{g,h}: \Sigma M_g \to \Sigma M_h$ via $\beta^{g,h}([p,v]) = [b^{g,h}p,v]$ and $\widetilde{\beta}^{g,h}([p,v]) = [\widetilde{b}^{g,h}p,v]$. It clearly also holds that $\beta^{g,h}(X) \cdot \widetilde{\beta}^{g,h}(\psi) = \widetilde{\beta}^{g,h}(X\cdot \psi)$ for all vector fields $X \in \Gamma(TM)$ and spinor fields $\psi \in \Gamma(\Sigma M_g)$.

We readily extend $\tilde{\beta}^{g,h}$ to spinors twisted by a vector bundle $E$ by defining $\widetilde{\beta}^{g,h} (\psi_x \otimes \xi_x) = \widetilde{\beta}^{g,h}(\psi_x) \otimes \xi_x$ for all $\xi_x \in E_x$, $\psi_x \in \Sigma_x M_g$ and all $x \in M$.

\subsection{Conformal Invariance}
For a given metric $g$ we denote by
\begin{align*}
    \mathcal{S}^{DYM}_g(\omega, \Psi) = \int_{M} |F_{\omega}|_g^2 + \ip{\D_{\omega} \Psi}{\Psi}_g\ \diff v_{g}
\end{align*}
the corresponding Dirac-Yang-Mills functional on $\Gamma(\Sigma M_g \otimes E) \times \mathcal{C}(P)$. If $\overline{g} = e^{2u}g$ for some smooth function $u \in \smooth{M}$ we shall use the shorthand $\overline{X} := \beta^{g,\overline{g}}(X) = e^{-u}X$ and $\overline{\psi}:= \widetilde{\beta}^{g,\overline{g}}(\psi) \in \Gamma(\Sigma M_{\overline{g}})$ for vector fields $X \in \Gamma(TM)$ and spinors $\psi \in \Gamma(\Sigma M_g)$. 
The Dirac-Yang-Mills action then exhibits the following transformation behaviour under a conformal transformation of the metric in dimension 4.
\begin{lemma}\label{lem: conformal behaviour}
Let $(M,g)$ be a $4$-dimensional Riemannian spin manifold, $\omega \in \mathcal{C}(P)$ a connection one form on $P$ and $\Psi \in \Gamma(\Sigma M \otimes E)$ an $E$-valued spinor. The for a conformal transformation $\overline{g} = e^{2u}g, \ u \in \smooth{M}$ of the metric on $M$ we have
\begin{align*}
    \mathcal{S}^{DYM}_g(\omega, \Psi) = \mathcal{S}^{DYM}_{\overline{g}}(\omega, e^{-\frac{3}{2}u}\overline{\Psi}).
\end{align*}
\end{lemma}

\begin{proof}
Since $\dim M = 4$ have $|F_\omega|^2_{\overline{g}} = e^{-4u}|F_\omega|^2_{g}$ while $v_{\overline{g}} = e^{4u}\diff v_g$. Thus, $|F_\omega|^2_{\overline{g}}v_{\overline{g}} = |F_\omega|^2_{g}\diff v_g$. For the term involving $\Psi$ we have the well known relation (see eg Proposition 1.3.10 of \cite{Gi})

$$
\overline{\D_\omega}\left(e^{-\frac{3}{2}u}\overline{\Psi}\right) = e^{-\frac{4}{2}u}\overline{\D_\omega \Psi}.
$$
Here $\overline{\D_\omega}$ denotes the Dirac operator on $\Gamma(\Sigma M_{\overline{g}}\otimes E)$. Since the bundle isomorphism $\Psi \mapsto \overline{\Psi}$ is also an isometry we have $$\ip{e^{-\frac{3}{2}u}\overline{\Psi}}{\overline{\D_\omega}\left(e^{-\frac{3}{2}u}\overline{\Psi}\right)}\diff v_{\overline{g}} = \ip{\Psi}{\D_\omega \Psi}\diff v_g$$
which concludes the proof.
\end{proof}

This leads to the following result generally referred to as conformal invariance of the Dirac-Yang-Mills system. 
\begin{theorem}\label{thm: conformal invariance}
Let $(M,g)$ be a $4$-dimensional closed Riemannian spin manifold, $u \in \smooth{M}$, and $\overline{g} = e^{2u}g$ a conformal transformation of the metric. If $(\omega,\Psi)$ is a Dirac-Yang-Mills pair with respect to the metric $g$ then $(\omega, e^{-\frac{3}{2}u}\overline{\Psi})$ is a Dirac-Yang-Mills pair with respect to $\overline{g}$.
\end{theorem}
\begin{proof}
    By Lemma \ref{lem: conformal behaviour} $\mathcal{S}^{DYM}_{\overline{g}}(\omega, e^{-\frac{3}{2}u}\overline{\Psi}) = \mathcal{S}^{DYM}_g(\omega, \Psi)$ and since the latter is a critical point so is the former. Hence by Proposition \ref{prop: EL-equations} $(\omega, e^{-\frac{3}{2}u}\overline{\Psi})$ solve the Dirac-Yang-Mills equations with respect to the transformed metric $\overline{g}$.
\end{proof}

\subsection{The stress-energy tensor}
We now consider a smooth variation $h(t)$ of the metric with $h(0) = g$ and
$$
\dtatzero h(t) = k, 
$$
where $k$ is a symmetric $(2,0)$ tensor.

The \emph{stress-energy tensor} of the Dirac-Yang-Mills system is defined by
\begin{align}\label{eq: stress-energy definition}
    \dtatzero\mathcal{S}^{DYM}_{h(t)}(\omega, \Psi) = \int_M \ip{T^{\mathrm{DYM}}}{k}_g \diff v_g,
\end{align}
where $\ip{\cdot}{\cdot}_g$ denotes the standard inner product on $T^*M\otimes T^*M$ induced by $g$.
\begin{theorem}\label{thm: DYM stress-energy tensor}
    The stress-energy tensor $T^\mathrm{DYM}$ has the form
    $$
    T^\mathrm{DYM} = T^\mathrm{Dirac} + T^\mathrm{YM},
    $$
    where $T^\mathrm{Dirac}$ and $T^\mathrm{YM}$ are symmetric $(2,0)$-tensors given by
    $$
    T^\mathrm{Dirac}(X,Y) = -\frac{1}{4}\ip{X\cdot \widetilde{\nabla}^\omega_Y \Psi + Y \cdot \widetilde{\nabla}^\omega_X \Psi}{\Psi} + \frac{1}{2}\ip{\D_\omega \Psi}{\Psi}g(X,Y)
    $$
    and
    $$
    T^\mathrm{YM}(X,Y) = -\ip{X\intprodl F_\omega}{Y\intprodl F_\omega} + \frac{1}{2}|F_\omega|^2 g(X,Y),
    $$
    respectively, for all vector fields $X,Y \in \Gamma(TM)$
\end{theorem}
\begin{proof}
    Let $h(t)$ be a smooth variation of the metric with $h(0) = g$ and 
    $\dtatzero h = k \in \mathrm{Sym}(2,0)$. We define $T^\mathrm{Dirac}$ and $T^\mathrm{YM}$, respectively by
    $$
        \dtatzero\int_M |F_\omega|^2_g \diff v_g = \int_M \ip{T^\mathrm{YM}}{k}_g \diff v_g,
    $$
    and
    $$
        \dtatzero\int_M \ip{\D_\omega \Psi}{\Psi}_g \diff v_g = \int_M \ip{T^\mathrm{Dirac}}{k}_g \diff v_g.
    $$
    The formula for $T^\mathrm{YM}$ then follows from Proposition 2.1 of \cite{Ba}. 
    Turning to the Dirac part we suppress the connection in the notation for the twisted Dirac operator. Instead for this part we shall denote by $\D_{h(t)}$ the Dirac operator twisted by $\nabla^\omega$ acting on sections of $\Sigma M_{h(t)} \otimes E$. Setting $\Psi_{h(t)} = \tilde{\beta}^{g,h(t)} \Psi$ we have
    $$
        \int_M \ip{T^\mathrm{Dirac}}{k}_g \diff v_g = \int_M \ip{\D_g \Psi}{\Psi}_g \dtatzero \left [dv_{h(t)}\right] + \int_M \dtatzero\left[\ip{\D_{h(t)}\Psi_{h(t)}}{\Psi_{h(t)}}_{h(t)} \right] \diff v_g.
    $$
    We recall
    \begin{align}\label{eq: metric var volume form}
        \dtatzero dv_{h(t)} = \frac{1}{2}\ip{g}{k}_g \diff v_g.     
    \end{align}
    For the second term, it follows from Theorem 21 of \cite{BoGa} and a simple computation that
    $$
    \dtatzero \tilde{\beta}^{h(t),g}\D_{h(t)}\Psi_{h(t)} = \frac{1}{4}\left( \div_g(k_g) + d\tr_g(k_g)\right)\cdot \Psi - \frac{1}{2}\sum_{i=1}^m e_i \cdot \widetilde{\nabla}^\omega_{k_g(e_i)}\Psi.
    $$
    This yields
    \begin{align}\label{eq: metric var Dirac}
        \begin{split}
        \dtatzero\ip{\D_{h(t)}\Psi_{h(t)}}{\Psi_{h(t)}}_{h(t)} &= \dtatzero \ip{\tilde{\beta}^{h(t),g}\D_{h(t)}\Psi_{h(t)}}{\Psi}_{g}\\
        &= -\frac{1}{2}\ip{\sum_{i=1}^m e_i \cdot \widetilde{\nabla}^\omega_{k_g(e_j)} \Psi}{\Psi}\\
        &= - \frac{1}{4}\ip{\sum_{i,j=1}^m k_{ij} \left(e_i \cdot \widetilde{\nabla}^\omega_{e_j} + e_j \cdot \widetilde{\nabla}^\omega_{e_i}\right)}{\Psi}\\
        &= \ip{T_{\Psi,g}}{k}_g,
        \end{split}
    \end{align}
    with $T_{\Psi,g}(X,Y) = -\frac{1}{4}\ip{X\cdot \widetilde{\nabla}^\omega_Y \Psi + Y \cdot \widetilde{\nabla}^\omega_X \Psi}{\Psi}$. Here $k_{ij} := k(e_i,e_j) = g(e_i, k_g(e_j))$. The formula for $T^\mathrm{Dirac}$ then follows from the equations \eqref{eq: metric var volume form} and \eqref{eq: metric var Dirac}.
\end{proof}

We note that in this paper we are only considering smooth sections of vector bundles, so that in particular a critical point corresponds to a solution to the Euler-Lagrange equations in the classical sense.
\begin{theorem}\label{thm: trace and divergence of S-E tensor}
    At a critical point $(\omega,\Psi)$ of the Dirac-Yang-Mills functional, the stress-energy tensor $T^\mathrm{DYM}$ satisfies
    $$
    \div (T^\mathrm{DYM}) = 0
    $$
    and
    $$
    \tr (T^\mathrm{DYM}) = \left(\frac{m}{2}-2\right)|F_\omega|^2.
    $$
    In particular, if $m =4 $ then $\tr (T^\mathrm{DYM}) = 0$.
\end{theorem}
\begin{proof}
    We have
    \begin{align*}
        \tr(T^\mathrm{YM}) = \frac{m}{2}|F_\omega|^2 - \sum_{i=1}^m \ip{e_i\intprodl F_\omega}{e_i \intprodl F_\omega} = \frac{m}{2}|F_\omega|^2 - 2|F_\omega|^2,
    \end{align*}
    while
    \begin{align*}
        \tr(T^\mathrm{Dirac}) = \frac{m}{2}\ip{\D_\omega \Psi}{\Psi} - \frac{1}{2} \sum_{i=1}^m \ip{e_i \cdot \widetilde{\nabla}_{e_i}^\omega\Psi}{\Psi} = \frac{m-1}{2}\ip{\D_\omega \Psi}{\Psi} = 0
    \end{align*}
    at a critical point.
    Now we assume without loss of generality that at a point the orthonormal frame $\{e_i\}_{i=1}^m$ satisfies $\nabla_{e_i}e_j = [e_i,e_j] = 0$ for $i,j \in \{1, \dots, m\}$. We denote by $R, R^{\Sigma M}$ and $\widetilde{R}_\omega$ the curvature tensors of the Levi-Civita connection and the connections $\nabla^{\Sigma M}$ and $\widetilde{\nabla}^\omega$, respectively.
    For the divergence of $T^\mathrm{YM}$ we have, applying the Bianchi identity to Proposition 2.3 of \cite{Ba}
    \begin{align*}
        \div(T^\mathrm{YM})(e_k) &= \ip{\delta_\omega F_\omega}{e_k \intprodl F_\omega}\\
        &= \ip{J(\Psi)}{e_k \intprodl F_\omega}\\
        &= -\frac{1}{2}\ip{K_{e_k \intprodl F_\omega}\Psi}{\Psi}\\
        &= -\frac{1}{2}\sum_{i=1}^m \ip{e_i \cdot \rho_* F_\omega(e_k,e_i)\Psi}{\Psi}\\
        &= -\frac{1}{2}\left(\sum_{i=1}^m \ip{e_i \cdot \widetilde{R}_\omega(e_k,e_i)\Psi}{\Psi} - \ip{e_i \cdot R^{\Sigma M}(e_k,e_i)\Psi}{\Psi}\right)\\
        &= -\frac{1}{2}A + \frac{1}{2}B
    \end{align*}
    with
    $$
    A = \sum_{i=1}^m \ip{e_i \cdot \widetilde{R}_\omega(e_k,e_i)\Psi}{\Psi}
    \quad \textrm{and} \quad
    B = \sum_{i=1}^m\ip{e_i \cdot R^{\Sigma M}(e_k,e_i)\Psi}{\Psi}.
    $$
    Then 
    \begin{align*}
        B = \sum_{i,s,t =1}^m \ip{g(R(e_k,e_i)e_s, e_t) e_i\cdot e_s \cdot e_t \cdot \Psi}{\Psi}
    \end{align*}
    and it follows from the properties of Clifford multiplication and the symmetries of the Riemann curvature tensor $R$ that
    $$
    B = -2B = 0.
    $$
    For $A$ we have
    \begin{align*}
        A &= \ip{\widetilde{\nabla}^\omega_{e_k}\D_\omega \Psi}{\Psi} -\ip{\D_\omega\widetilde{\nabla}_{e_k}^\omega\Psi}{\Psi}\\
        &= -\ip{\D_\omega\widetilde{\nabla}_{e_k}^\omega\Psi}{\Psi}.
    \end{align*}

    For $\div(T^\mathrm{Dirac})$ we obtain
    \begin{align*}
        \div(T^\mathrm{Dirac})(e_k) &= -\frac{1}{4}\left(\sum_{i=1}^m \ip{e_i \cdot \widetilde{\nabla}^\omega_{e_i} \widetilde{\nabla}^\omega_{e_k} \Psi + e_k \cdot \widetilde{\nabla}^\omega_{e_i} \widetilde{\nabla}^\omega_{e_i} \Psi}{\Psi} + \ip{e_i \cdot \widetilde{\nabla}_{e_k}^\omega\Psi + e_k \cdot \widetilde{\nabla}^\omega_{e_i} \Psi}{\widetilde{\nabla}^\omega_{e_i} \Psi}\right)\\
        &= - \frac{1}{4}\left(\ip{\D_\omega\widetilde{\nabla}^{\omega}_{e_k}\Psi}{\Psi} -\ip{e_k \cdot \Delta_\omega \Psi}{\Psi} -\ip{\widetilde{\nabla}^\omega_{e_k} \Psi}{\D_\omega \Psi}  + \sum_{i=1}^m \ip{e_k \cdot \widetilde{\nabla}^\omega_{e_k} \Psi}{\widetilde{\nabla}^\omega_{e_k} \Psi}\right)\\
        &= - \frac{1}{4}\ip{\D_\omega\widetilde{\nabla}^{\omega}_{e_k}\Psi}{\Psi} + \frac{1}{4}\ip{e_k \cdot \Delta_\omega \Psi}{\Psi}.
    \end{align*}
    In the final step the third term vanishes due to $\D \Psi = 0$ and the fourth by the properties of Clifford multiplication.
    Applying the Weizenb\"ock formula to the second term yields
    \begin{align*}
        \ip{e_k \cdot \Delta_\omega \Psi}{\Psi} &= \ip{e_k\cdot \D_\omega^2\Psi}{\Psi} - \frac{1}{4}S_g\ip{e_k \cdot \Psi}{\Psi} - \ip{e_k \cdot\rho_*(F_\omega) \cdot \Psi}{\Psi}.
    \end{align*}
    The terms involving the scalar curvature and $\D_\omega^2$ vanish and expanding the remaining term we are left with
    \begin{align*}
        &\ip{e_k \cdot \Delta_\omega \Psi}{\Psi}\\
        &\quad = -\frac{1}{2}\ip{\sum_{i,j=1}^m e_k\cdot e_i\cdot e_j \cdot \rho_* F_{\omega}(e_i,e_j)\Psi}{\Psi}\\
        &\quad = -\frac{1}{2}\left[\ip{\sum_{\substack{i=1\\j=k}}^m e_i \cdot \rho_* F_{\omega}(e_i,{e_k})\Psi}{\Psi} - \ip{\sum_{\substack{j=1\\i=k}}^m e_j \cdot \rho_* F_{\omega}(e_k,{e_i})\Psi}{\Psi} + \ip{\sum_{i,j \neq k}^m e_k\cdot e_i\cdot e_j\cdot \rho_* F_\omega(e_i,e_j)\Psi}{\Psi}\right]\\
        &\quad = -\ip{\sum_{i=1}^m e_i \cdot \rho_* F_{\omega}(e_i,{e_k})\Psi}{\Psi}  -\frac{1}{2} \ip{\sum_{i,j \neq k}^m e_k\cdot e_i\cdot e_j\cdot \rho_* F_\omega(e_i,e_j)\Psi}{\Psi}.
        \end{align*}
    The operator $e_k\cdot e_i\cdot e_j\cdot \rho_* F_\omega(e_i,e_j)$ with $i \neq j\neq k$ is skew-Hermitian with respect to the bundle metric on $\Sigma M \otimes E$ so the second term in the final line vanishes. For the first term, we recall from the computation of $\div(T^{YM})$
    \begin{align*}
        -\ip{\sum_{i=1}^m e_i \cdot \rho_* F_{\omega}(e_i,{e_k})\Psi}{\Psi} =  A = - \ip{\D_\omega\widetilde{\nabla}^\omega_{e_k} \Psi}{\Psi}.
    \end{align*}
    Hence we have,
    \begin{align*}
        \div(T^\mathrm{Dirac})(e_k) =  -\frac{1}{4} \ip{\D_\omega\widetilde{\nabla}^\omega_{e_k} \Psi}{\Psi}  -\frac{1}{4} \ip{\D_\omega\widetilde{\nabla}^\omega_{e_k} \Psi}{\Psi} =  -\frac{1}{2} \ip{\D_\omega\widetilde{\nabla}^\omega_{e_k} \Psi}{\Psi}.
    \end{align*}
    Thus, at a critical point
    \begin{align*}
        \div(T^\mathrm{DYM})(X) = \div(T^\mathrm{YM})(X) + \div(T^\mathrm{Dirac})(X) = 0 
    \end{align*}
    for all $X \in \Gamma(TM)$.

\end{proof}

\section{Vanishing of the Current \& Characterisation of Decoupling Connections}\label{sec: perturbation theory}
This section is devoted to proving Theorems \ref{thm: int-vanishing on chiral}, \ref{thm: int-decoupling classification}, \ref{thm: int-perturbation minimal floor} and \ref{thm: int-generic invertibility theorem}. To accomplish this we apply the tools of analytic perturbation theory to the Dirac-Yang-Mills system to study when the Dirac-Current $J(\Psi)$ vanishes on the kernel of the Dirac operator associated to a given connection. This approach, particularly \ref{thm: variational criteria} is inspired by the work done by Bernd Ammann and Nicolas Ginoux in \cite{AmGi} on Dirac-harmonic maps.
We first derive a variational criteria for the current to vanish.

\begin{theorem}\label{thm: variational criteria}
Let $\omega \in \mathcal{C}(P)$ and $\Psi \in \mathrm{ker}(\D_\omega) \subset \Gamma(\Sigma M \otimes E)$ a twisted harmonic $E$-valued spinor. Then $J(\Psi) = 0$ if and only if for each affine variation $\omega_t = \omega + t\eta$, $\eta \in \Omega^1(M;\ad(P))$, $t\in(-\epsilon,\epsilon),\epsilon>0$, of $\omega$ there exists a smooth variation $(\Psi_t)_{t\in (-\epsilon,\epsilon)}$ of $\Psi$ such that
$$
\dtatzero (\D_{\omega_t}\Psi_t,\Psi_t)_{L^2} = 0.
$$

\end{theorem}
\begin{proof}
    We have
    \begin{align}\label{eq: L^2 derivative}
        \begin{split}
              \dtatzero (\D_{\omega_t}\Psi_t,\Psi_t)_{L^2} &= \int_M \dtatzero\ip{\D_{\omega_t}\Psi_t}{\Psi_t} \diff v_g\\
            &= \int_M \ip{\dtatzero\D_{\omega_t}\Psi_t}{\Psi} +  \ip{\D_{\omega}\Psi}{\dtatzero\Psi_t}\diff v_g\\
            &= \int_M \ip{\dtatzero\D_{\omega_t}\Psi_t}{\Psi} \diff v_g.      
        \end{split}
    \end{align}
    Now let $\phi_t = \Psi_t-\Psi$. Then from the relation \eqref{eq: difference of Diracs} and the definition \eqref{eq: definition of K_eta} of $K_\eta$ 
    \begin{align*}
        \D_{\omega_t}\Psi_t = \D_{\omega}\Psi_t + tK_{\eta}\Psi_t = \D_{\omega} \Psi + \D_{\omega}\phi_t + tK_{\eta}\Psi + tK_{\eta}\phi_t = \D_{\omega}\phi_t + tK_{\eta}\Psi + tK_{\eta}\phi_t,
    \end{align*}
    so that
    \begin{align*}
        \dtatzero\D_{\omega_t}\Psi_t = \D_{\omega} \phi + K_{\eta}\Psi,
    \end{align*}
    where $\phi = \dtatzero \phi_t$.
     Inserting this into equation \eqref{eq: L^2 derivative} we have by assumption
    \begin{align*}
        0 &= \int_M \ip{\dtatzero\D_{\omega_t}\Psi_t}{\Psi}  \diff v_g\\
        &= \int_M \ip{\D_{\omega} \phi}{\Psi} + \ip{K_{{\eta}}\Psi}{\Psi}  \diff v_g\\
        &= \int_M \ip{ \phi}{\D_{\omega}\Psi} - 2\ip{J(\Psi)}{{\eta}} \diff v_g\\
        &= -2\int_M \ip{J(\Psi)}{{\eta}} \diff v_g.
    \end{align*}
    As this must hold for all ${\eta} \in \Omega^1(M;\ad(P))$ we conclude that $J(\Psi) = 0$ as desired.
\end{proof}
\begin{corollary}\label{thm: uncoupled solutions}
Let $(\omega,\Psi)$ be a pair of a Yang-Mills connection $\omega$ and a twisted harmonic $E$-valued spinor $\Psi \in \ker(\D_\omega) \subset \Gamma(\Sigma M \otimes E)$ satisfying the variational condition of Theorem \ref{thm: variational criteria}.Then $(\omega,\Psi)$ is an uncoupled Dirac-Yang-Mills pair.
\end{corollary}

\begin{corollary}\label{cor: chiral spinor vanishing result}
Suppose $\dim(M)$ is even and let $\Sigma M = \Sigma^+ M  \oplus \Sigma^- M$ be the chiral decomposition of $\Sigma M$. If $\Psi \in \Gamma(\Sigma^+ M\otimes E)$ or $\Psi \in \Gamma(\Sigma^- M \otimes E)$, then $J(\Psi) = 0$.
\end{corollary}
This result will enable an application of index theory to proving existence of decoupled solutions in the next section.
\begin{remark}
    The above Corollary \ref{cor: chiral spinor vanishing result} holds also in the case when $M$ is not closed. This is best seen by noting that, for $\Psi \in \Sigma^{\pm}M \otimes E$ 
    and any $X \in \Gamma(TM)$ and $\sigma \in \Gamma(\ad(P))$, we have
    $$
    \ip{\Psi}{X\cdot \rho_*(\sigma)\Psi} = 0.
    $$
    Thus all coefficients in the expression \eqref{eq: Dirac current} vanish and $J(\Psi) = 0$.
\end{remark}
We shall now work towards a classification of those connection one-forms for which the condition of Theorem \ref{thm: variational criteria} is satisfied for all $E$-valued harmonic spinors.

We call a connection one-form $\omega \in \mathcal{C}(P)$ \emph{perturbation critical} if for each affine variation $\omega_t = \omega + t\eta$ and each eigenvalue $\lambda(t)$ of $\D_{\omega_t}$ such that $\lambda(0) = 0$, we also have that $\lambda(t)$ is differentiable at $0$ with $\lambda'(0) = 0$. 
\begin{theorem}\label{thm: decoupling classification}
A connection one-form $\omega \in \mathcal{C}(P)$ is decoupling if and only if it is perturbation critical.
\end{theorem}
Recall that by decoupling we mean that $J(\Psi) = 0$ for all $\Psi \in \ker \D_\omega$.
For the proof of Theorem \ref{thm: decoupling classification} we will need the following lemma.

\begin{lemma}\label{lem: analytic prop of affine perturbations}
    Let $\omega \in \mathcal{C}(P)$ be a connection one-form and fix an affine variation $\omega_t = \omega + t\eta$, $t\in \rn$, 
    where $\eta \in \Omega^1(M;\ad(P))$. Then the eigenvalues $\lambda_{n}(t)$ (counted with multiplicity) of the associated Dirac operators $\D_{\omega_t}$ are real analytic in $t$. Furthermore, there exists a sequence $\{\phi_l(t)\}_l$ of holomorphic $\Gamma(\Sigma M \otimes E)$-valued functions on $\rn$ such that for each $t\in \rn$ $\{\phi_l(t)\}_l$ is a complete orthonormal family of eigenvectors with respect to the eigenvalues $\lambda_{n}(t)$.
\end{lemma}

\begin{proof}
    For the Dirac operator induced by $\omega_t$  $\D_{\omega_t}$ we have $\D_{\omega_t} = \D_\omega + tK_{\eta}$. We consider the complexification $\D_{z} = \D_\omega + zK_{\eta}$, $z \in \cn$ of this affine perturbation of $\D_0 = \D_\omega$ as a family of operators on $L^2(M; \Sigma M \otimes E)$. We shall show that this perturbation constitutes a \emph{self-adjoint holomorphic family} satisfying all the requirements of \cite[Theorem 3.9 (p.392)]{Ka}
    from which the existence and analyticity of $\lambda_n(t)$ and $\phi_l(t)$ follows. 

    \noindent
    {\bf Claim A:} The operator $K_\eta$ is a self-adjoint bounded operator defined on all of $L^2(M; \Sigma M \otimes E)$.

    \noindent
    {\it Proof of Claim A:}
    The operator $K_\eta$ is a smooth section of the bundle $\textrm{End}(\Sigma M \otimes E) \cong (\Sigma M \otimes E)^* \otimes \Sigma M \otimes E$. Thus for $x \in M$ we have $(K_\eta \Psi)(x) = (K_\eta)_x(\Psi_x)$ which allows for extending $K_\eta$ to all of $L^2(M; \Sigma M \otimes E)$. Furthermore the operators $(K_\eta)_x$ are clearly all bounded and by compactness 
    $$
    |(K_\eta)_x| \leq \sup_{y\in M} |(K_\eta)_y|  = \max_{y\in M} |(K_\eta)_y|<\infty.
    $$
    From this we obtain 
    \begin{align*}
        \|K_\eta \Psi\|_{L^2}^2 &= \int_M |(K_\eta \Psi)(x)|^2 
        \ \diff v_g(x)\\
        &\leq \int_M |(K_\eta)|^2 |\Psi_x|^2 \ \diff v_g (x) \\
        &\leq (\max_{y\in M} |(K_\eta)_y|)^2 \|\Psi\|_{L^2}^2,
    \end{align*}
    for $\Psi \in L^2(M; \Sigma M \otimes E)$. We have that $K_\eta$ is self-adjoint, 
    from the relations \eqref{eq: difference of Diracs} and \eqref{eq: definition of K_eta}.
    which concludes the proof of Claim A.

    \noindent
    {\bf Claim B}: $\D_z$ is a closed operator with compact resolvent.

    \noindent
    {\it Proof of Claim B:}
    That $\D_z$ is closed follows from self-adjointedness.
    In order to show the resolvent is compact, it is sufficient to consider the case $z = 0$ and check that the resolvent $\mathrm{Res}(\D_\omega, i) = (\D_\omega - iI)^{-1}$ is compact. 
    
    \noindent
    From the identity $\D_z^* = \D_{\bar{z}}$, the power series expansion $\D_z = \D_\omega + zK_\eta$ and Claim A, it follows that $\D_z$ is a self-adjoint holomorphic family of type (A) as defined in \cite{Ka} (p. 375).  Together with Claim B this suffices for us to apply the aforementioned Theorem 3.9 in \cite{Ka}. The assertion $\phi_l(t) \in \Gamma(\Sigma M \otimes E)$ follows by ellipticity of $\D_{\omega_t}$. 
\end{proof}

\begin{proof}[Proof of Theorem \ref{thm: decoupling classification}]
    We first show that perturbation critical connections are decoupling. Let $\omega \in \mathcal{C}(P)$ be perturbation critical and fix an affine variation $\omega_t = \omega + t\eta$. Let $k = \dim \ker (\D_\omega)$ and take $\Psi \in \ker(\D_\omega)$. Denote by $ |\lambda_1(t)| \leq \cdots \leq |\lambda_k(t)|$ the eigenvalues of $\D_{\omega_t}$ which tend to $0$ as $t\to 0$. 
    By Lemma \ref{lem: analytic prop of affine perturbations} there is a family $\{\Psi_l(t)\}_{l=1}^k $ of holomorphic $\Gamma(\Sigma M \otimes E)$-valued functions such that for $t\neq 0$, $\{\Psi_{l}(t)\}_{l=1}^{k}$ is an orthonormal set of eigenspinors of $\D_{\omega_t}$ with $\D_{\omega_t}\Psi_{l}(t) = \lambda_{l}(t)\Psi_{l}(t)$ and such that $\{\Psi_{l}(0)\}_{l=1}^{k}$ is an orthonormal basis for $\ker(\D_{\omega})$. 

    \noindent
    The section $\Psi \in \ker(\D_\omega)$ can be expressed as
    \begin{align*}
        \Psi = \sum_{l=1}^k a_{l}\Psi_{l}(0)
    \end{align*}
    for some $a_l \in \cn, \ l =1,\dots,k$. We define a one-parameter variation $\Psi(t)$ of $\Psi$ via
    \begin{align*}
        \Psi(t) = \sum_{l=1}^k a_{l}\Psi_{l}(t).
    \end{align*}
    Clearly $\Psi(t)$ is a smooth variation with $\Psi(0) = \Psi$. 
    We have, by the assumption of perturbation criticality,
    \begin{align*}
        \dtatzero\varip{\D_{\omega_t}\Psi_t}{\Psi_t}_{L^2} &= \dtatzero\left[\sum_{l=1}^k a_l^2 \lambda_{l}(t)\right] = \sum_{l=1}^k a_l^2 \lambda_l'(0) = 0.
    \end{align*}
    From Theorem \ref{thm: uncoupled solutions} we conclude that $\omega$ is decoupling.

    \noindent
    Now assume $\omega$ is not perturbation critical. Then there exists some $\eta \in \Omega^1(M;\ad(P))$ and some eigenvalue $\lambda(t)$ of $\D_{\omega_t} = \D_{\omega} + tK_{\eta}$ such that $\lambda(0) = 0$ while $\lambda'(0) \neq 0$. By Lemma \ref{lem: analytic prop of affine perturbations} there exists a family $\Psi(t)$ of smooth sections of $\Sigma M \otimes E$, depending smoothly (in fact holomorphically) on $t$ and such that $\D_{\omega_t} \Psi(t) = \lambda(t)\Psi(t)$. Set $\Psi = \Psi(0) \in \ker(\D_\omega)$ and compute
    \begin{align}
        \dtatzero\varip{\D_{\omega_t}\Psi(t)}{\Psi(t)}_{L^2} = \dtatzero \varip{\lambda(t)\Psi(t)}{\Psi(t)}_{L^2} = \lambda'(0) \neq 0.
    \end{align}
    Hence by Theorem \ref{thm: uncoupled solutions} $J(\Psi) \neq 0$.
    This concludes the proof.
\end{proof}
We shall now show that this set of decoupling connections contains a dense and open subset of $\mathcal{C}(P)$

We call $\omega$ \emph{perturbation minimal} if there exists a $C^\infty$-neighbourhood $U\subset \mathcal{C}(P)$ of $\omega$ such that for $\vartheta \in U$, $\dim \ker (\D_{\vartheta}) \geq \dim \ker (\D_{\omega})$.

\begin{corollary}\label{cor: pert-min decoupling}
    Let $\omega \in \mathcal{C}(P)$ be a perturbation minimal connection one-form. Then $\omega$ is decoupling.
\end{corollary}
\begin{proof}
By Theorem \ref{thm: decoupling classification} it is sufficient to note that perturbation minimal connections are also perturbation critical. In fact fixing an affine variation $(\omega_t)_{t\in (-\epsilon,\epsilon)}$, upper semi-continuity of the function $t \mapsto \dim \mathrm{ker}\D_{\omega_t}$ guarantees that each eigenvalue tending to zero along the variation $\omega_t$ is constantly equal to $0$ for all sufficiently small values of $t$. 
\end{proof}
\begin{proposition}\label{prop: perturbation minimal dense}
    The set $\mathcal{C}_{\mathrm{pm}}(P)$ of perturbation minimal connections is dense and open in $\mathcal{C}(P)$ with respect to the $C^\infty$ topology.
\end{proposition}
\begin{proof}
    Take $\omega \in \mathcal{C}(P)$ fixed but arbitrary and fix an open neighbourhood $U \subset \mathcal{C}(P)$ of $\omega$. The function $\omega \mapsto \dim \mathrm{ker}\D_\omega$ is integer-valued and bounded from below on $U$. Hence it attains its minimum at some connection $\vartheta\in U$. By definition $\vartheta$ is then perturbation minimal.
    From Corollary 4.13 in \cite{Te} we have that the spectrum of $\D_\omega$ depends continuously on $\omega$, which implies that $\mathcal{C}_{\mathrm{pm}}(P)$ is open.
\end{proof}

A further consequence of Lemma \ref{lem: analytic prop of affine perturbations} is the following useful property of perturbation minimal connections.

\begin{proposition}\label{prop: pert-min dimension}
    Let $\omega$ and $\vartheta$ be perturbation minimal connection one-forms on $P$. Then $\dim  \mathrm{ker}\D_\omega = \dim  \mathrm{ker}\D_{\vartheta}$.
\end{proposition}

\begin{proof}
    Let $\eta \in \Omega^1(M;\ad(P))$ be the difference $\eta = \vartheta-\omega$. Consider the affine family of connections $\omega_t = \omega + t\eta$ for $t\in I = (0,1)$. By Lemma \ref{lem: analytic prop of affine perturbations} the eigenvalues of $\D^{\omega_t}$ are analytic functions on $I$. Further, by perturbation minimality, $\lambda(0) = 0$ if and only if $\lambda(t) \equiv 0$ on some neighbourhood of $0$, and likewise for $\lambda(1) = 0$. Thus $\lambda(0) = 0$ if and only if $\lambda(t) \equiv 0$ on $I$ and in particular, then $\lambda(1) = 0$. Likewise if $\lambda(1) = 0$ we must have $\lambda(0) = 0$ as well. This concludes the proof.
\end{proof}

We shall denote this constant dimension of the kernel of the Dirac operator on $E$ associated to a perturbation minimal connection $\omega \in \mathcal{C}(P)$ by $k_{\mathrm{pm}}(E)$ 
This acts as a lower bound of the dimension of the solution space of the Dirac-Yang-Mills equations for a fixed Yang-Mills connection form $\omega$ in the sense of the following result:
\begin{theorem}\label{thm: perturbation minimal floor}
    Let $\omega \in \mathcal{C}(P)$ and let $\D_\omega$ be the associated 
    Dirac operator acting on $E$-valued spinors. Suppose further that $k_{\mathrm{pm}}(E) > 0$. Then there exists a subspace $V \subset \mathrm{ker}\D_\omega$ of dimension at least $k_{\mathrm{pm}}(E)$ such that $J(\Psi)=0$ for all $\Psi \in V$. In particular, for each Yang-Mills connection $\omega$ there is a $k_{\mathrm{pm}}(E)$-dimensional space of solutions $(\omega, \Psi)$ to the Dirac-Yang-Mills equations with connection $\omega$.
\end{theorem}
\begin{proof}
Choose a perturbation minimal connection $\vartheta$ and consider the affine variation $\omega_t = \omega +t\eta$, where $\eta$ is given by $\vartheta = \omega + \eta$ and $t \in (-\epsilon,\epsilon)$. We fix $\epsilon>0$ small enough such that $\dim \mathrm{ker}\D_{\omega_t}$ is constant for $t\in (-\epsilon,\epsilon) \setminus \{0\}$. Then, for $t\in (-\epsilon,\epsilon)\setminus \{0\}$, $\omega_t$ is perturbation minimal and by Proposition \ref{prop: pert-min dimension} we have $\dim \mathrm{ker}\D_{\omega_t} = k_{\mathrm{pm}}(E)$. By Lemma \ref{lem: analytic prop of affine perturbations} there is a family $\{\Psi_i (t)\}_{i = 1}^{k_{\mathrm{pm}}(E)}$ of smooth sections of $\Sigma M \otimes E$ depending holomorphically on $t$ and such that for $t\in (-\epsilon,\epsilon)\setminus \{0\}$ $\{\Psi_1(t),\dots \Psi_{k_{\mathrm{pm}}(E)}(t)\}$ is an orthonormal basis for $\ker( \D_{\omega_t})$.

Now set $V_t = \mathrm{span}\{\Psi_1(t),\dots \Psi_{k_{\mathrm{pm}}(E)}(t)\}$ and $V = V_0$. Then $V \subset \mathrm{ker}\D_{\omega}$ and clearly each $\phi \in V$ is the limit $\lim_{t\to 0} \phi_t$ of a variation $\phi_t$ with $\phi_t \in V_t$. By Corollary \ref{cor: pert-min decoupling}, $J(\phi_t) = 0$ for each $t>0$ so that, by continuity of the current, $J(\phi) = 0$. Thus $J$ vanishes on $V$ which concludes the proof.
\end{proof}

\subsection{Implications of non-vanishing current}

We have shown that the inclusion $\ker \D_\omega \subset \ker J$ holds for a generic connection one-form $\omega$. While the primary application in this paper is to the case where the kernel of $J$ is non-trivial, the case $\ker J = \{0\}$ gives an interesting obstruction result:
\begin{proposition}\label{prop: generic vanishing}
    If $\ker J = \{0\}$ then each decoupling connection is perturbation minimal and $$\dim \ker \D_\omega = 0$$
    for each such connection one-form $\omega$.
\end{proposition}
As an example we have the following:
\begin{theorem}\label{thm: generic invertibility theorem}
Let $(M^3,g)$ be a closed $3$-dimensional Riemannian spin manifold endowed with an $G$-principal bundle $G\to P\to M$, where $G$ is either $\SU{N}$ or $\U{N}$ and let $E = \asb{P}{\cn^N}{\mathrm{st}}$. Then each decoupling connection one-form $\omega$ is perturbation minimal and the associated Dirac operator on $\Gamma(\Sigma M \otimes E)$ satisfies 
\begin{equation*}
    \dim \ker \D_\omega = 0.
\end{equation*}
\end{theorem}
\begin{proof}
Take $\Psi \in \Gamma(\Sigma M \otimes E)$ such that $J(\Psi) = 0$. Note that $\Spin{3} \cong \SU{2}$ and that the spinor representation in this case coincides with the standard representation. On a contractible open neighbourhood $U$ of $M$ we fix a local section $e = (e_1,e_2,e_3): U \subset M \to \SO{M}$, a choice of associated local section $\epsilon: U \subset M \to \Spin{M}$ of $\Spin{M}$ and a local gauge $p: U \subset M \to P$. We can then locally write a spinor $\Psi \in \Gamma(\Sigma M \otimes E)$ as
\begin{equation*}
    \Psi = [\epsilon\times p, \sum_{j = 0}^1\sum_{k=0}^{N-1} \Psi^{jk}\varepsilon^{(2)}_j\otimes\varepsilon^{(N)}_k],
\end{equation*}
where $(\varepsilon^{(n)}_0,\dots, \varepsilon^{(n)}_{n-1})$ denotes the standard orthonormal basis of $\cn^n$.
For $k = 1,\dots,3$ set $\gamma_k = -i\sigma_k$, where
\begin{equation*}
    \sigma_1 =
    \begin{bmatrix}
        0&1\\
        1&0
    \end{bmatrix},\quad
    \sigma_2 =
    \begin{bmatrix}
        0&-i\\
        i&0
    \end{bmatrix},\quad
    \sigma_3 =
    \begin{bmatrix}
        1&0\\
        0&-1
    \end{bmatrix}
\end{equation*}
are the Pauli matrices. Then $(\gamma_1,\gamma_2,\gamma_3)$ satisfy Clifford relations and so generate the Clifford algebra $\mathrm{Cl}(3)$. Hence we can choose $e$ and $p$ such that
\begin{equation*}
    e_\ell \cdot \mathrm{st}_*(\sigma)\Psi = [\epsilon\times p, \sum_{j,k} \Psi^{jk}\gamma_\ell\varepsilon^{(2)}_j\otimes\sigma\varepsilon^{(N)}_k].
\end{equation*}
It will be most convenient to choose the $\Ad$-invariant inner product on $\SU{N}$ (or $\U{N}$) to be 
\begin{equation*}
    \ip{Z}{W} = -2\tr(ZW).
\end{equation*}
Denote by $E_{rs}$ the $N\times N$ matrix with $(E_{rs})_{jk} = \delta_{jr}\delta_{ks}$. For each integer $1\leq k \leq N-1$ the elements $Y_{k} = -\frac{1}{2}(E_{0k} - E_{k0})$, $X_{k} = -\frac{i}{2}(E_{0k} + E_{k0})$ and $H_k = -\frac{i}{2}(E_{00} - E_{kk})$ are orthogonal and span a subalgebra of $\su{N}$ isomorphic to $\su{2}$. We extend these to an orthonormal basis $\mathfrak{B}_k$ of $\SU{N}$ (or $\U{N})$. Since $J(\Psi) = 0$ we have in particular
\begin{align*}
    0 &= \ip{e_3\cdot \mathrm{st}_*(H_k)\Psi}{\Psi} = \frac{1}{2}(\lvert\Psi^{0k}\rvert^2 + \lvert\Psi^{10}\rvert^2 - \lvert\Psi^{00}\rvert^2 - \lvert\Psi^{1k}\rvert^2),\\
    0 &= \ip{e_1\cdot \mathrm{st}_*(X_k)\Psi}{\Psi} = \re(-\Psi^{00}\overline{\Psi}^{1k} - \Psi^{10}\overline{\Psi}^{0k}),\\
    0 &= \ip{e_2\cdot \mathrm{st}_*(Y_k)\Psi}{\Psi} = \re(\Psi^{00}\overline{\Psi}^{1k} - \Psi^{10}\overline{\Psi}^{0k}),\\
    0 &= \ip{e_1\cdot \mathrm{st}_*(Y_k)\Psi}{\Psi} = \im(\Psi^{00}\overline{\Psi}^{1k} + \Psi^{10}\overline\Psi^{0k}),\\
    0 &= \ip{e_2\cdot \mathrm{st}_*(X_k)\Psi}{\Psi} = \im(\Psi^{00}\overline\Psi^{1k} - \Psi^{10}\overline\Psi^{0k}),\\
    0 &= \ip{e_1\cdot \mathrm{st}_*(H_k)\Psi}{\Psi} = \re(\Psi^{0k}\overline\Psi^{1k}-\Psi^{00}\overline\Psi^{10}),\\
    0 &= \ip{e_2\cdot \mathrm{st}_*(H_k)\Psi}{\Psi} = \im(\Psi^{00}\overline\Psi^{10} - \Psi^{0k}\overline\Psi^{1k}),\\
    0 &= \ip{e_3\cdot \mathrm{st}_*(Y_k)\Psi}{\Psi} = \im(\Psi^{00}\overline\Psi^{0k} - \Psi^{10}\overline\Psi^{1k}),\\
    0 &= \ip{e_3\cdot \mathrm{st}_*(X_k)\Psi}{\Psi} = \re(\Psi^{10}\overline\Psi^{1k}-\Psi^{00}\overline\Psi^{0k}).
\end{align*}
Here the bar denotes complex conjugation.
From these equations one can deduce that $\Psi^{00} = \Psi^{0k} = \Psi^{10} = \Psi^{1k} \equiv 0$ on $U$. By varying $k$ we show that all components of $\Psi$ vanish identically on $U$. Covering $M$ by contractible neighbourhoods gives $\Psi \equiv 0$ on $M$. The conclusion then follows from Proposition \ref{prop: generic vanishing}.
\end{proof}
\begin{remark}
    The set of metric connections on a Hermitian vector bundle is in one-to-one correspondence with the set of connection one-forms on its unitary frame bundle. Theorem \ref{thm: generic invertibility theorem} thus shows that, for any hermitian vector bundle $E$ over a spin $3$-manifold $M$ the property that the associated Dirac operator has trivial kernel holds for a generic metric connection on $E$.
\end{remark}

\section{Existence Results via Index Theory}\label{sec: index theory}
For an even-dimensional Riemannian spin manifold $(M^{2n},g)$ we have the chiral decomposition of the spinor bundle $\Sigma M = \Sigma^+M \oplus \Sigma^-M$ inducing a chiral decomposition of the twisted bundle $\Sigma M \otimes E$. For any connection $\omega \in \mathcal{C}(P)$ the associated Dirac operator interchanges the chiral subspaces and thus we may write
\begin{align*}
    \D_\omega = 
    \begin{pmatrix}
    0 & \D^+_\omega\\
    \D^-_\omega & 0
    \end{pmatrix}.
\end{align*}
Self-adjointedness of $\D_\omega$ translates into the relation $(\D^+_\omega)^* = \D^-_\omega$. On a closed spin manifold, the Dirac operator and it's restrictions to the chiral subspaces are Fredholm, and we have
\begin{align*}
    \mathrm{ind}(\D^+_\omega) = \dim \ker(\D^+_\omega) - \dim \ker(\D^-_\omega).
\end{align*}

The Atiyah-Singer index theorem of \cite{AtSi} gives a formula for this index in terms of characteristic classes of $TM$ and $E$ as
\begin{align*}
    \mathrm{ind}(\D^+_\omega) = \int_M \hat{A}(TM)\ch(E).
\end{align*}
Here $\hat{A}(TM),\ch(E) \in H_\mathrm{dR}^\mathrm{even}(M)$ are the $\hat{A}$-genus of $TM$ and the Chern character of $E$, respectively. The index of the untwisted Dirac operator $\slashed{\partial}^+:\Gamma(\Sigma M) \to \Gamma(\Sigma M)$ is given by
\begin{align*}
    \mathrm{ind}(\slashed{\partial}^+) = \hat{A}[M]:= \int_M \hat{A}(TM).
\end{align*}
An important consequence of this is that the index is a topological invariant and, in particular, independent of the metric $g$ on $M$ and the connection $\omega$ on $P$. For this reason we shall use the notation $\mathrm{ind}(\D^+_E)$ to mean the index of a Dirac operator twisted by any metric affine connection $\nabla$ on $E$.

Chern-Weil theory gives a method for calculating the characteristic classes appearing in the index theorem in terms of curvature forms of connections on the associated bundles. In particular the Chern character $\ch(E)$ is the de Rham-class represented by the closed form
\begin{align}\label{eq: chern repr}
    \ch(\nabla^\omega) = \tr\left(\exp\left(\frac{i}{2\pi}\rho_*(F_\omega)\right)\right) = \sum_{k = 1}^\infty \frac{i^k}{(2\pi)^k}\tr\left(\rho_*(F_\omega)^{\wedge k}\right).
\end{align}

Here $\rho_*(F_\omega)^{\wedge k}$ is just the iterated wedge product of $\rho_*(F_\omega)$ with itself, i.e. $\rho_*(F_\omega)^{\wedge 0} = 1$ and $\rho_*(F_\omega)^{\wedge k+1} = \rho_*(F_\omega)^{\wedge k}\wedge \rho_*(F_\omega)$. Of course $\rho_*(F_\omega)^{\wedge k} = 0$ for $k > n$ and we have
\begin{align*}
    \ch(E) = \sum_{k=0}^n \ch_k(E),
\end{align*}
with $\ch_k(E) \in H_{dR}^{2k}(M)$ represented by
\begin{align}\label{eq: k-th chern}
    \ch_k(\nabla^\omega) = \frac{i^k}{(2\pi)^k}\tr\left(\rho_*(F_\omega)^{\wedge k}\right).
\end{align}
There is a similar expansion
\begin{align}
    \hat{A}(TM) = \sum_{k=0}^{\floor{\frac{n}{2}}} \hat{A}_k(TM)
\end{align}
with $\hat{A}_k (TM) \in H_{dR}^{4k}$. In particular $\hat{A}_0(TM) = 1$ and $\ch_0(E) = \mathrm{rank} \ E$. The relevance of the index theorem to the Dirac-Yang-Mills system is evident by the following consequence of Corollary \ref{cor: chiral spinor vanishing result}. 
\begin{proposition}\label{prop: existence in terms of Index}
    Let $(M^{2n},g)$ be an even-dimensional Riemannian spin manifold, $G\to P \to M$ a principal bundle with compact structure group $G$ and $\rho: G \to \U{V}$ a unitary representation giving rise to the associated Hermitian vector bundle $E = \asb{P}{V}{\rho}$. If $\mathrm{ind}(\D^+_E) = k \neq 0$ then for any Yang-Mills connection $\omega \in \mathcal{C}(P)$ there is an $\ell$-dimensional subspace $W \subset \ker{\D_\omega}$ with $\ell\geq |k|$ such that $(\omega,\Psi)$ is an uncoupled Dirac-Yang-Mills pair for each $\Psi \in W$. 
\end{proposition}
\begin{proof}
    If $k > 0$ then $\ell = \dim \ker(\D^+_\omega) \geq k > 0$. Setting $W = \ker(\D^+_ \omega)$ By Corollary \ref{cor: chiral spinor vanishing result} $J(\Psi) = 0$ for each $\Psi \in W$ and hence for such $\Psi$, $(\omega,\Psi)$ is an uncoupled Dirac-Yang-Mills pair. If $k<0$ we instead have $\ell = \dim \ker(\D^-_\omega) \geq -k > 0$ and we may take $W = \ker(\D^-_\omega)$.
\end{proof}

As the following well known result shows, the index only provides information when we have $\dim M \equiv 0 \ (\mathrm{mod}\ 4)$.
\begin{proposition}\label{prop: index 0 if not 4n}
    If $m = \dim M \equiv  2 \ (\textrm{mod } 4)$ then $\mathrm{ind}(\D^+_\omega) = 0$.
\end{proposition}
\begin{proof}
    Let $E^*$ be the Hermitian dual bundle to the Hermitian vector bundle $E$. The idea is to now show that when $\dim M \equiv 2 \ (\textrm{mod } 4)$ we have $\mathrm{ind}(\D^+_{E}) = \mathrm{ind}(\D^+_{E^*}) = -\mathrm{ind}(\D^+_{E})$.
    
    Letting $\nabla$ be any metric connection on $E$ and $\nabla^*$ the dual connection, the curvature tensors $F_\nabla$ and $F_{\nabla^*}$ of $\nabla$ and $\nabla^*$ are related by
    $$
    F_\nabla = -(F_{\nabla^*})^t
    $$
    with the transpose acting on the $\End(E)$ factor. Inserting this into Equation \eqref{eq: k-th chern} gives the relation
    $$
    \ch_k(E^*) = (-1)^{k}\ch_k(E).
    $$
    For $m = 4\ell$ we then have
    \begin{align*}
        \left[\hat{A}(TM)\ch(E^*)\right]_{4\ell} = \ch_{2\ell}(E) + \hat{A}_1(TM)\ch_{2\ell-2}(E) + \cdots + N\hat{A}_\ell(TM) 
    \end{align*}
    while for $m = 4\ell + 2$ we get
    \begin{align*}
        \left[\hat{A}(TM)\ch(E^*)\right]_{4\ell + 2} = -\ch_{2\ell +1}(E) - \hat{A}_1(TM)\ch_{2\ell - 1}(E) - \cdots - \hat{A}_{\ell}(TM)\ch_1(E).
    \end{align*}
    Inserting this into the index theorem then gives,
    \begin{align}\label{eq: dual index}
        \mathrm{ind}(\D^+_{E^*}) =
        \begin{cases}
        \mathrm{ind}(\D^+_{E}) & \textrm{if } m \equiv 0 \ (\textrm{mod } 4)\\
        -\mathrm{ind}(\D^+_{E}) & \text{if } m \equiv 2 \ (\textrm{mod } 4).
        \end{cases}
    \end{align}
    Denote by $\flat: E \to E^*$ the map $\xi^\flat(\zeta) = \ip{\zeta}{\xi}$. Note that this map is an \emph{anti}-linear bijection. We may extend this to a map $\Sigma M \otimes E\to \Sigma M \otimes E^*$ by letting $\flat$ act only on the second factor. This new map clearly preserves chirality and one can also show it to commute with the Dirac operator in the sense that
    $$
    \D_{\nabla^*}\Psi^{\flat} = \left(\D_{\nabla}\Psi\right)^{\flat}.
    $$
    Hence, for all values of $m = \dim M$, $\mathrm{ind}({\D^+_{E^*}}) = \mathrm{ind}({\D^+_{E}})$. Comparing with the relation \eqref{eq: dual index} we obtain $\mathrm{ind}(\D^+_{E}) = 0$ whenever $m = 2 (\textrm{mod } 4)$.
\end{proof}
If $M$ satisfies $\dim (M) = 4$ and $\hat{A}[M] \neq 0$ and we take as our structure group $\SU{N}$, we can always produce a representation yielding a non-zero index. For this we first need the following comparison lemma.
\begin{lemma}\label{lem: index st vs index ad}
    Let $(M,g)$ be a 4-dimensional Riemannian spin manifold with fixed spin structure and $\SU{N} \to P \to M$ a principal fibre bundle over $M$. Set $E = \asb{P}{\cn^N}{\mathrm{st}}$. Then
    $$
    \mathrm{ind}(\D^+_{\ad(P)}) = 2N \ \mathrm{ind}(\D^+_{E}) - (N^2+1) \ \mathrm{ind}(\slashed{\partial}^+).
    $$
\end{lemma}
\begin{proof}
    We first note that the index theorem as stated above requires a complex twisting bundle, whereas $\ad(P)$ is real. However, by the isomorphism $\Sigma M \otimes_\rn \ad(P) \cong \Sigma M \otimes_\cn \ad(P)^\cn$ of complex vector bundles we can apply the index theorem to the twisting bundle $\ad(P)^\cn$ with typical fibre $\su{N}\otimes_\rn \cn \cong \sl{N}{\cn}$ instead.
    Let $\omega \in \mathcal{C}$ be any smooth connection one-form on $P$. We have the decomposition $F_\omega = F_\omega^+ + F_\omega^-$ into self-dual and anti-self-dual components. Take $\{\alpha^\ell\}_{\ell=1}^3$, $\{\beta^\ell\}_{\ell=1}^3$ to be local frames of $\Omega_+^2(M)$ and $\Omega_-^2(M)$ respectively, satisfying 
    $$
        \alpha^i \wedge\alpha^j = \delta^{ij}\diff v_g
    $$
    and 
    $$
        \beta^i\wedge\beta^j = -\delta^{ij}\diff v_g.
    $$
    Then 
    $$
        F_{\omega}^+ = \sum_{\ell = 1}^3 \alpha^\ell \otimes F^+_\ell
    $$
    and 
    $$
        F_{\omega}^- = \sum_{\ell = 1}^3 \beta^\ell \otimes F^-_\ell.
    $$
    We compute, with $B_{\sl{N}{\cn}}(X,Y) = 2N\tr(XY)$ the Killing form of the Lie algebra $\sl{N}{\cn}$,
    \begin{align*}
        \ch_2(\ad(P)^\cn) &= \frac{-1}{4\pi^2}\tr(\ad(F_\omega)\wedge \ad(F_\omega))\\
        &= \frac{-1}{4\pi^2}\sum_{\ell = 1}^3 \left(\tr(\ad(F_\ell^+)\ad(F_\ell^+)) - \tr(\ad(F_\ell^-)\ad(F_\ell^-))\right)v_{g}\\
        &= \frac{-1}{4\pi^2}\sum_{\ell = 1}^3 \left((B_{\sl{N}{\cn}}(F_\ell^+,F_\ell^+) - B_{\sl{N}{\cn}}(F_\ell^-,F_\ell^-)\right)v_{g}\\
        &= \frac{-2N}{4\pi^2}\sum_{\ell = 1}^3 \left((\tr((F_\ell^+)^2) - \tr((F_\ell^-)^2) \right)v_{g}\\
        &=  2N \ \ch_2(E).
    \end{align*}
    From the Index Theorem we have 
    \begin{align*}
        \mathrm{ind}(\D^+_{\ad(P)}) &= \int_M \ch_2(\ad(P)^\cn) + (\dim \su{N}) \mathrm{ind}(\slashed{\partial}^+)\\
        & = 2N \int_M \ch_2(E) + (N^2 -1)\mathrm{ind}(\slashed{\partial}^+)\\
        & = 2N \ \mathrm{ind}(\D^+_{E}) -2N^2 \mathrm{ind}(\slashed{\partial}^+) + (N^2 -1)\mathrm{ind}(\slashed{\partial}^+)\\
        & = 2N \ \mathrm{ind}(\D^+_{E}) - (N^2+1) \ \mathrm{ind}(\slashed{\partial}^+).
    \end{align*}
\end{proof}

From this the following result, establishing part $(i)$ of Theorem \ref{thm: int-A-hat nonzero}, is a straightforward consequence.
\begin{theorem}\label{thm: 2 implies 3}
Let $(M,g)$ be a $4$-dimensional Riemannian spin manifold with fixed spin structure and $\SU{N} \to P \to M$ a principal fibre bundle over $M$. Set $E = \asb{P}{\cn^N}{\mathrm{st}}$. Then any two of the following implies the third
\begin{enumerate}[label = (\roman*)]
    \item $\ind(\slashed{\partial}^+) = 0$,
    \item $\ind(\D_E^+) = 0$,
    \item $\ind(\D_{\ad(P)}^+) = 0$.
\end{enumerate}
In particular if $\hat{A}[M] = \ind(\slashed{\partial}^+) \neq 0$ then at least one of $\ind(\D_E^+)$ and $\ind(\D_{\ad(P)}^+)$ is nonzero.
\end{theorem}

\begin{example}
A family of complex closed spin manifolds with $\hat{A}[M] \neq 0$ is given by the non-singular algebraic complex projective hypersurfaces $V^{2n}(2d) \subset \cn\mathrm{P}^{2n+1}$ of even degree $2d$. More precisely, $V^{n}(d)$ is the set of zeroes of a homogeneous degree $d$ polynomial $f(z_0,\dots,z_n)$ in $n+1$ complex variables satisfying $(\nabla f)(z) \neq 0$ whenever $f(z) = 0$ and $z\neq 0$.

We have $\dim(V^{2n}(2d)) = 4n $ and any two such hypersurfaces with equal dimension and degree are diffeomorphic. Furthermore, $V^{n}(d)$ is spin if and only if $n+d$ is even and the $\hat{A}$-genus of $V^{2n}(2d)$ is (\cite{LaMi} p.138)
\begin{equation*}
    \hat{A}[V^{2n}(2d)] = \frac{2d}{(2n+1)!}\prod_{k=1}^n (d^2-k^2)
\end{equation*}
so that $ \hat{A}[V^{2n}(2d)] > 0$ for $d>n$. In particular this holds for all of the $4$-manifolds $V^2(2d)$ with $d > 1$. The case $d = 2, n = 1$ is the well known $K3$ surface.

Let $(E,h)$ be a Hermitian vector bundle over $V^2(2d)$ with $c_1(E) = 0$. Each unitary connection $\nabla$ on $E$ satisfying $F_\nabla^{0,2} = F_\nabla^{2,0}$ induces a holomorphic structure on $E$ via $\overline{\partial}_E := \pi^{0,1} \circ \nabla$. By this we mean that $\overline{\partial}_E: \Gamma(E) \to \Omega^{0,1}(M;E)$ satisfies $\overline{\partial}_E^2 = 0$ and $\overline{\partial}_E (f\xi) = (\overline{\partial}f)\otimes \xi + f(\overline{\partial}_E\xi)$ for $f \in \smooth{M,\cn}$ and $\xi \in \Gamma(E)$. A holomorphic subbundle of $(E,\overline{\partial}_E)$ is a complex subbundle $F$ such that $\restr{\overline{\partial}_E}{F}$ is a holomorphic structure on $F$.

In \cite{Donaldson} it is shown that a connection $\nabla$ on a Hermitian vector bundle $E$ over a complex algebraic surface $M \subset \cp^{N}$ is an irreducible Hermitian-Einstein connection, and hence an irreducible Yang-Mills connection, if and only if it induces a holomorphic structure $\overline{\partial}_E$ on $E$ such that $(E,\overline{\partial}_E)$ is stable. Under the assumption $c_1(E) = 0$ stability is equivalent to having no proper non-trivial holomorphic subbundles. This gives the following:
\end{example}
\begin{corollary}\label{cor: existence on complex surf}
    Let $E \to V^2(2d)$ be an orientable Hermitian vector bundle of rank $N$ and let $\SU{N}\to P\to V^2(2d)$ be its oriented unitary frame bundle. Let $\nabla$ be a unitary connection on $E$ inducing a stable holomorphic structure and let $\omega \in \mathcal{C}(P)$ be the corresponding connection one-form. Then either
    \begin{equation*}
        \{(\omega,\Psi) \ | \ \Psi \in \ker \D_{E,\omega}^+ \cup \ker \D_{E,\omega}^-\}
    \end{equation*}
    or 
    \begin{equation*}
        \{(\omega,\Psi) \ | \ \Psi \in \ker \D_{\ad(P),\omega}^+ \cup \ker \D_{\ad(P),\omega}^-\}
    \end{equation*}
    is a non-empty family of uncoupled Dirac-Yang-Mills pairs.
\end{corollary}

\begin{proof}
    Since $\hat{A}[V^2(2d)] \neq 0$ Theorem \ref{thm: 2 implies 3} yields a non-zero index for one of the bundles $E$ and $\ad(P)$. Proposition \ref{prop: existence in terms of Index} together with Theorem 1 of \cite{Donaldson} then provides a non-empty family of Dirac-Yang-Mills pairs for that bundle.
\end{proof}

If $M$ is a 4-manifold with $\hat{A}[M] = 0$ Theorem \ref{thm: 2 implies 3} does not give any information. This is the case for example if $M$ admits a metric of positive scalar curvature or, by virtue of the multiplicative property of the $\hat{A}$-genus,
 $M = M_1\times M_2$ for a pair of spin manifolds $M_1$ and $M_2$

In the $\hat{A}[M] = 0$ case we have the following sufficient conditions.
\begin{theorem}\label{thm: instantons give solutions v2}
    Let $(M^4,g)$ be a $4$-dimensional Riemannian spin manifold with $\hat{A}[M] = 0$ and let $G \to P \to M$ be a non-trivial principal bundle with compact structure group and $\rho: G \to \U{V}$ a faithful  unitary representation, giving rise to a Hermitian vector bundle $E = \asb{P}{V}{\rho}$. If $P$ admits an (anti-)self-dual connection one-form $\vartheta$ then $k=\mathrm{ind}(\D^+_E) \neq 0$ and for any Yang-Mills connection there exist a $|k|$-dimensional space of uncoupled Dirac-Yang-Mills pairs $(\omega,\Psi)$.
\end{theorem}
\begin{proof}
    The data guarantees that $E$ is itself non-trivial and since $\hat{A}[M] = 0$ the index of $\D^+_E$ is given by
    \begin{align*}
        \mathrm{ind}(\D^+_E) = \int_M \ch(E) = \int_M \ch_2(E).
    \end{align*}
    Take $\{\alpha^\ell,\beta^\ell\}_{\ell = 0}^3$ as in the proof of Lemma \ref{lem: index st vs index ad}
    so that in the case of $\star F_\omega = F_\omega$ we may write
    $$
    F_\omega = \sum_{\ell =1}^3 \alpha^\ell \otimes F_\ell.
    $$
    Non-triviality guarantees that the coefficients $F_\ell$ cannot all vanish everywhere.
    Inserting this into Formula \eqref{eq: k-th chern} for $\ch_2(E)$ gives
    \begin{align*}
        \ch_2(E) &= \frac{-1}{4\pi^2}\tr(\rho_*(F_{\widetilde{\omega}})\wedge \rho_*(F_{\widetilde{\omega}}))\\ 
        &= \frac{-1}{4\pi^2}\sum_{\ell = 1}^3 \tr(\rho_*(F_\ell)^2)\diff v_g\\
        &= \frac{1}{4\pi^2}\sum_{\ell = 1}^3 \tr(\rho_*(F_\ell) \overline{\rho_*(F_\ell)}^T)\diff v_g.
    \end{align*}
    In the last line we have used that locally $\rho_*(F_\ell)$ is given by a $\u{N}$-valued function. Thus $\ch_2(E)$ is represented by a non-vanishing and everywhere non-negative multiple of the volume form and we have 
    \begin{align*}
        \mathrm{ind}(\D^+_E) = \int_M \ch_2(E) > 0.
    \end{align*}
    The anti-self-dual case is obtained similarly by working with $\{\beta^\ell\}_{\ell =1}^3$ instead.
\end{proof}
In particular, if $M$ admits a metric of positive scalar curvature, we have:
\begin{proposition}\label{prop: Lichnerowicz for twisted}
    Let $M$ be a $4$-dimensional closed spin manifold, let $G\to P \to M$ be a principal bundle with compact structure group and let $\rho: G\to \U{V}$ be a unitary representation giving rise to the associated Hermitian vector bundle $E = \asb{P}{V}{\rho}$. If $M$ admits a metric $g$ with $S_g > 0$ with respect to which $P$ admits an (anti-)self-dual connection $\omega \in \mathcal{C}(P)$, then 

    \begin{align*} 
    \mathrm{ind}(\D^+_E) = \int_M \ch(E) = 
        \begin{cases}
            \dim \ker (\D^+_{g,\omega}) = \dim \ker (\D_{g,\omega})$ \textrm{ if } $\star F_\omega = F_\omega,\\
            -\dim \ker (\D^-_{g,\omega}) = -\dim \ker (\D_{g,\omega})$ \textrm{ if } $\star F_\omega = -F_\omega.
        \end{cases}
    \end{align*}
\end{proposition}
\begin{proof}
In this case
    \begin{align*}
        \mathrm{ind}(\D^+_E) = \int_M \ch(E).
    \end{align*}
    For a chiral twisted spinor $\Psi^\pm \in \Gamma(\Sigma^\pm M \otimes E)$ and a two-form $\alpha \in \Omega^2(M)$ we have
    $$
    \star \alpha \cdot \Psi^\pm = \pm\alpha\cdot \Psi^\pm
    $$
    so that $\alpha \cdot \Psi^- = 0$ if $\alpha$ is self-dual and $\alpha\cdot\Psi^+ = 0$ if $\alpha$ is anti-self dual. Assuming $\star F_\omega = F_\omega$ the Weizenb\"ock formula \eqref{eq: Weizenbock} then gives
    \begin{align*}
        (\D_\omega^-)^2\Psi^- &= \Delta_\omega\Psi^- + \frac{1}{4}S_g \Psi^- + \rho_*(F_\omega)\cdot \Psi^-\\
        &= \Delta_\omega\Psi^- + \frac{1}{4}S_g \Psi^-
    \end{align*}
    for $\Psi^- \in \Gamma(\Sigma^- M \otimes E)$. Then since $S_g > 0$ $\ker{\D_\omega^-} = \{0\}$.
    If $\star F_\omega = -F_\omega$ we instead have
    \begin{align*}
        (\D_\omega^+)^2\Psi^+ &= \Delta_\omega\Psi^+ + \frac{1}{4}S_g \Psi^+
    \end{align*}
    so that $\ker{\D_\omega^+} = \{0\}$.
\end{proof}

So that in this case the bound on $\md{E}$ given by the index is sharp:
\begin{corollary}\label{cor: instantons are perturbation minimal}
    If $\dim M = 4$ and $S_g > 0$ then each (anti-)self-dual connection $\omega \in \mathcal{C}(P)$ is perturbation minimal and $\md{E} = \left|\mathrm{ind}(\D^+_\omega)\right|$.
\end{corollary}

\subsection{\texorpdfstring{$\SU{2}$}{SU(2)}-bundles}
We now restrict our attention to the case where we have an $\SU{2}$-bundle $\SU{2}\to P \to M$ and $\dim M = 4n$ for some $n\in \mathbb{Z}_{>0}$. We recall (see for instance Section 4.2 of \cite{Ha}) that for $\ell \in \mathbb{Z}_{>0}$ the unique irreducible representation $(\rho_\ell,V_\ell)$ of $\SU{2}$ of dimension $\ell+1$ is given by the $\ell$-th symmetric tensor power of the standard matrix representation:
$$
V_\ell = S^\ell(\cn^2) = S^\ell(V_1).
$$
On the level of vector bundles, we set $E = \asb{P}{\cn^2}{\mathrm{st}}$ and 
$$
E_\ell = \asb{P}{V_\ell}{\rho_\ell} = S^\ell(E).
$$
By making use of the splitting principle (see \cite{LaMi} Proposition 11.1 p. 225) and the fact that $c_1(E) = 0$ since $E$ is orientable we can then compute $\ch(E_\ell)$ in terms of $c_2(E)$
as
\begin{equation*}
    \ch(E_\ell) = \sum_{j=0}^n p_j(\ell)(c_2(E))^j,
\end{equation*}
where $p_j \in \mathbb{Q}[\ell]$ is a degree $2j+1$ polynomial explicitly given by
\begin{equation*}
    p_j(\ell) = \frac{(-1)^j}{(2j)!}\sum_{s=0}^\ell (\ell -2s)^{2j}.
\end{equation*}
This yields a formula for the index of $\D^+_{E_\ell}$ in terms of $c_2(E)$ and $\hat{A}(TM)$.
\begin{proposition}\label{prop: index of SU(2)-bundles}
    Fix $(M^{4n},g)$ and $\SU{2}\to P \to M$ and set $E = \asb{P}{\cn^2}{\mathrm{st}}$. Then
    \begin{equation*}
        \mathrm{ind}(\D_{E_\ell}^+) = \sum_{j=0}^n p_j(\ell)a_\ell,
    \end{equation*}
    where 
    \begin{equation*}
        a_\ell = a_\ell(M,E) = \hat{A}_{n-\ell}(TM)(c_2(E))^\ell[M] = \int_M \hat{A}_{n-\ell}(TM)(c_2(E))^\ell.
    \end{equation*}
    In particular $a_0 = \hat{A}(M) = \mathrm{ind}(\slashed{\partial}^+)$ is independent of $E$.
\end{proposition}
This also shows the index to be a rational polynomial, i.e. $\mathrm{ind}(\D_{E_\ell}^+) = q(\ell)$ for some $q \in \mathbb{Q}[\ell]$, of degree at most $2n-1$.

Setting $a(M,E) = (a_0,\dots,a_n) \in \mathbb{Q}^n$ as in Proposition \ref{prop: index of SU(2)-bundles} then gives
\begin{theorem}\label{thm: existence SU(2)}
    Let $(M^{4n},g)$ be a $4n$-dimensional closed Riemannian spin manifold and $\SU{2}\to P \to M$ an $\SU{2}$-principal bundle over $M$ with $E = \asb{P}{\cn^2}{\mathrm{st}}$.
    \begin{enumerate}[label = (\roman*)]
        \item If $a(M,E) \neq 0$ then
    \begin{equation*}
        \big|\mathrm{ind}(\D_{E_l}^+)\big| > 0
    \end{equation*}
    holds for all except at most $2n-1$ choices of $\ell \in \mathbb{Z}_{>0}$. 
    \item If $q(\ell) =\mathrm{ind}(\D_{E_l}^+)$ has no positive integer roots then for each        connection $\omega \in \mathcal{C}(P)$ and associated vector bundle $E_\rho = \asb{P}{V}{\rho}$ with $\rho$ non-trivial
        \begin{equation*}
            \dim \ker (\D_{\omega,E_\rho}) > 0.
        \end{equation*}
    \end{enumerate}

\end{theorem}
\begin{proof}
    Part $(i)$ follows directly from Proposition \ref{prop: index of SU(2)-bundles} as $q(\ell) = \ind(\D^+_{E_\ell})$ is not the zero polynomial and has degree at most $2n-1$, and so at most that many positive integer roots.
    For $(ii)$ we decompose the representation $\rho$ into a direct sum of irreducible representations,
    \begin{equation*}
        E_\rho = E_{\ell_1} \oplus \cdots \oplus E_{\ell_m}
    \end{equation*}
    for some $m \in \mathbb{Z}_{>0}$. Set
    \begin{equation*}
        \Psi = \sum_{k=1}^m \Psi_k
    \end{equation*}
    with $\Psi_k \in \ker(\D_{\omega, E_{\ell_k}})$. Then $\D_{\omega,E_\rho}\Psi = 0$.
\end{proof}
\begin{corollary}
    Let $M^{4n}$, $P$ and $E$ be as in Theorem \ref{thm: existence SU(2)}. If $a(M,E) \neq 0$ and $P$ admits a Yang-Mills connection $\omega$. Then for all but at most $2n-1$ irreducible representations $\rho:\SU{2}\to V$, there exist non-trivial Dirac-Yang-Mills pairs $(\omega,\Psi)$ with $\Psi \in \Gamma(\Sigma M \otimes \asb{P}{V}{\rho})$.
\end{corollary}
If $\mathrm{ind}(\slashed{\partial}^+) = \hat{A}[M] = a_0 \neq 0$ the condition of part $(i)$ of Theorem \ref{thm: existence SU(2)} holds. As we have already seen, the complex hypersurfaces $V^{2n}(2d)$ with $d>n$ provide a family of such manifolds which we can enlarge by taking products and direct sums.
The hypersurfaces $V^{2n}(2d)$ also come with a natural K\"ahler structure induced by that on $\cp^{2n+1}$.

\begin{theorem}\label{thm: pairs on complex hypersurfaces}
    Let $M = V^{2n}(2d)$ for some $d>n>0$. Let $E \to M$ be an oriented Hermitian vector bundle with $\mathrm{rk}(E) = 2$ and set $P = \SU{E}$, the oriented unitary frame bundle of $E$. If $P$ admits a Yang-Mills connection $\omega$, then for all but at most $2n-1$ irreducible representations $\rho:\SU{2} \to V$ of $\SU{2}$ there exist non-trivial Dirac-Yang-Mills pairs $(\omega, \Psi)$ with $\Psi \in \Gamma(\Sigma M \otimes \asb{P}{V}{\rho})$.
\end{theorem}

In \cite{UhYa} a classification, generalising that of \cite{Donaldson}, of which Hermitian vector bundles over compact K\"ahler manifolds admit Yang-Mills connections can be found. Applying this to the case of $\mathrm{rank}$-2 orientable Hermitian vector bundles and combing with Theorem \ref{thm: pairs on complex hypersurfaces} yields an existence results for Dirac-Yang-Mills pairs in this setting.

\section{A Construction Method for Uncoupled Solutions and Explicit Examples}\label{sec: contructions} 
In this section we give an overview of a construction method for twisted harmonic spinors first developed on $\s^4$ (see \cite{Hi1}) which, by Corollary \ref{cor: chiral spinor vanishing result} also yields uncoupled Dirac-Yang-Mills pairs. 

In \cite{Hi1} this method is developed for $M = \s^4$ and while we will also only apply this construction to $\s^4$ and $\mathbb{T}^4$ we will present it first as an ansatz on a general even dimensional spin manifolds and with any choice of principal bundle and associated twisting bundle.

Explicitly, we take $(M^{2n},g)$ to be an even dimensional closed Riemannian spin manifold and we shall investigate pairs of the form $(\omega, \Theta \cdot \phi)$, where $\omega$ is a Yang-Mills connection on $P$, $\Theta \in \Omega^*(M;E)$ is an $E = P \times_\rho V$-valued form and $\phi \in \Gamma(\Sigma M)$ is a spinor on $M$.
We begin by deriving auxiliary equations for $\phi$ and $\Theta$ such that $\D_\omega(\Theta \cdot \phi) = 0$.
\begin{proposition}\label{prop: Dirac on clifford product}
Let $\omega$ be a connection one-form on $P$ and $\Psi = \Theta \cdot \phi$ for some $\Theta \in \Omega^*(M;E)$ and $\phi \in \Gamma(\Sigma M)$.
Then 
\begin{align}\label{eq: Twisted Dirac on Clifford product}
    \D_\omega\Psi = (D_\omega\Theta)\cdot \phi + D^\Theta\phi,
\end{align}
where $D_\omega = d_\omega + \delta_\omega$ and $D^\Theta: \Gamma(\Sigma M) \to \Gamma(\Sigma M \otimes E)$ is the first order differential operator given locally by
\begin{align*}
    D^\Theta\phi = \sum_{k=1}^m e_k \cdot \Theta \cdot \nab{\Sigma M}{e_k}{\phi}.
\end{align*}
\end{proposition}
bb
\begin{corollary}
For $\omega \in \mathcal{C}(P)$, $\phi \in \Gamma(\Sigma M)$ and $\Theta \in \Omega^*(M;E)$, if
\begin{align*}
    D_\omega \Theta = 0 \quad \textrm{and} \quad D^\Theta \phi = 0
\end{align*}
then
$$
\D_\omega(\Theta \cdot \phi) = 0.
$$
\end{corollary}
For the Proof of Proposition \ref{prop: Dirac on clifford product} we first recall some well-known identities in the following Lemma.
\begin{lemma}\label{lem: form and clifford identities}
    Let $M$ be a $m$-dimensional Riemannian manifold. Then for any one-form $\alpha$ and $\Theta \in \Omega^p(M;E)$, we have
    \begin{align*}
        \alpha \cdot \Theta = \alpha\wedge\Theta - *(\alpha \wedge * \Theta).
    \end{align*}
    and
    \begin{align*}
    d_\omega \Theta = \sum_{k=1}^{m} e^k \wedge \nab{\omega}{e_k}{\Theta}.
\end{align*}
\end{lemma}
\begin{proof}[proof of Proposition \ref{prop: Dirac on clifford product}]
    Using Lemma \ref{lem: form and clifford identities} we have
    \begin{align*}
        \D_{\omega}(\Theta \cdot \phi) &= \sum_{i=k}^{m}e_k\cdot \widetilde{\nabla}^\omega_{e_k}(\Theta \cdot \phi)\\
        &= \sum_{k=1}^{m} e_k \cdot \nab{\omega}{e_k}{\Theta} \cdot \phi + e_k \cdot\Theta \cdot \nab{\Sigma M}{e_k}{\phi}\\
        &= \sum_{k=1}^{m} (e^k \wedge \nab{\omega}{e_k}{\Theta} - \star e^k \wedge \nab{\omega}{e_k}\star \Theta) \cdot \phi + \sum_{k=1}^{m} e_k \cdot\Theta \cdot \nab{\Sigma M}{e_k}{\phi}\\
        &= (d_\omega \Theta - \star d_\omega\star \Theta)\cdot \phi + D^\Theta \phi\\
        &= (D_\omega \Theta) \cdot \phi +D^\Theta \phi. 
    \end{align*}
\end{proof}
A twistor spinor on a spin-manifold $(M^n,g)$ is a spinor $\phi \in \Gamma(\Sigma M)$ satisfying \cite{BauFrGrKa}:
\begin{align}\label{eq: twistor equation}
    \nab{\Sigma M}{X}{\phi} + \frac{1}{n}X\cdot \slashed{\partial}\phi = 0
\end{align}
for all vector fields $X \in \Gamma(TM)$. 
\begin{proposition}\label{prop: twistor spinors}
    Let $M$ be even dimensional so $\dim M = 2n$ and $\phi \in \Gamma(\Sigma M)$ be a twistor spinor on $M$. Take $\Theta \in \Omega^n(M;E)$. Then
    $$
        D^\Theta \phi = 0.
    $$
\end{proposition}
This result is a straightforward consequence of the following Lemma:
\begin{lemma}\label{lem: clifford conjugating forms}
    Let $M$ be a $m$-dimensional Riemannian manifold and $\Theta \in \Omega^r(M;E)$. Let $\{e_1, \dots, e_{m}\}$ be a local orthonormal frame for $TM$. Then
    \begin{align*}
        \sum_{i = 1}^m e_i \cdot \Theta \cdot e_i = (-1)^{r}(2r-m)\Theta.
    \end{align*}
    In particular, $m = 2n$ and $r = n$ gives
    \begin{align*}
        \sum_{i = 1}^{2n} e_i \cdot \Theta \cdot e_i = 0.
    \end{align*}
\end{lemma}
\begin{proof}
    By linearity of Clifford multiplication it is sufficient to establish the formula for r-forms of the form $\Theta = \xi \otimes e^{i_1} \wedge \cdots \wedge e^{i_r}$. for some $\xi \in \Gamma(E)$. In fact, as Clifford multiplication ignores the first factor it is sufficient to treat the case $\Theta = e^{i_1} \wedge \cdots \wedge e^{i_r} \in \Theta^r(M)$ (i.e. when $E$ is the trivial real line bundle over $M$).
    Let $I_{r} = \{i_1,\dots,i_r\}$. If $i \in I_r$ then
    $$
    e_{i}\cdot \omega \cdot e_{i} =(-1)^{r} \omega.
    $$
    Otherwise, if $i \notin I_{r}$, we have
    $$
    e_{i}\cdot \omega \cdot e_{i} =(-1)^{r+1}\omega.
    $$
    Thus,
    \begin{align*}
        \sum_{i = 1}^m e_i \cdot \omega \cdot e_i = \left((-1)^r\sum_{i \in I_{r}} \omega\right) -\left((-1)^r\sum_{i \notin I_{r}}\omega\right) = (-1)^r(2r-m)\omega.
    \end{align*}
\end{proof}
\begin{proof}[proof of Proposition \ref{prop: twistor spinors}]
    Let $\dim M = 2n$, $\phi \in \Gamma(\Sigma M)$ be a twistor spinor and $\Theta \in \Omega^n(M;E)$ a twisted $n$-form. Then
    \begin{align*}
        D^\Theta \phi &= \sum_{k = 1}^{2n} e_k \cdot \Theta \cdot \nab{\Sigma M}{e_k}{\phi}\\
        &= -\frac{1}{2n}\sum_{k = 1}^{2n} e_k \cdot \Theta \cdot e_k \cdot \slashed{\partial} \phi\\
        &= 0.
    \end{align*}
\end{proof}
\begin{theorem}\label{thm: main thm multiplication by forms}
Let $\omega$ be a Yang-Mills connection on $P \to M$. Then, if $M$ admits a non-trivial chiral twistor spinor $\phi \in \Gamma(\Sigma^{\pm} M)$ and an $E$-valued $n$-form $\Theta$ with
\begin{align*}
    D_\omega\Theta = 0, 
\end{align*}
the pair $(\omega,\Theta \cdot \phi)$ is an uncoupled solution to the Dirac-Yang-Mills equations.
\end{theorem}
\begin{proof}
    This follows immediately by Propositions \ref{prop: twistor spinors} and \ref{prop: Dirac on clifford product}
\end{proof}
\begin{corollary}\label{cor: construction of solution in dim 4}
    Let $\dim M = 4$ and let $\omega$ be a (non-flat) Yang-Mills connection on $G\to P\to M$ with $\ad(P)$-valued curvature two-form $F_{\omega}$. Then, if $M$ admits a non-trivial chiral twistor spinor $\phi \in \Gamma(\Sigma ^\pm M)$, the pair $(\omega, F_{\omega}\cdot\phi)$ is a Dirac-Yang-Mills pair with the coefficient bundle $E = \ad(P)$.
\end{corollary}
By the same procedure we can produce $\mathrm{End}(E)$-valued solutions to the Dirac equation, where $E$ is some Hermitian bundle associated to $P$ via a faithful representation $\rho$.

If we make the stronger assumption that $\psi \in \Gamma(\Sigma M)$ is parallel, then for any positive integer $r \leq n$ and $\Theta \in \Omega^{r}(M;E)$ we have $D^\Theta \psi = 0$. This gives a version of Corollary \ref{cor: construction of solution in dim 4} in arbitrary dimension
\begin{theorem}\label{thm: construction with parallels}
    Let $(M,g)$ be a closed Riemannian spin manifold of even dimension admitting a non-trivial parallel spinor field $\psi \in \Gamma(\Sigma M)$ with chiral decomposition $\psi =\psi^+ + \psi^-$. Let $\omega$ be a (non-flat) Yang-Mills connection on a principal fibre bundle $G \to P \to M$ with $\ad(P)$-valued curvature two-form $F_\omega$. Then $(\omega,F_\omega \cdot \psi^\pm)$ are Dirac-Yang-Mills pairs with coefficient bundle $E = \ad(P)$
\end{theorem}
A complete description of manifolds admitting parallel spinors can be found in \cite{Wang}. In particular, for every Calabi-Yau manifold $(M,g)$ the space of parallel spinors is $2$-dimensional and spanned by a pair $(\psi^+,\psi^-)$ of spinors of opposite chiralities. In this case the connection one-form $\omega_g$ of the Calabi-Yau metric $g$ also satisfies the Yang-Mills equations. Hence, we have:
\begin{theorem}\label{thm: DYM on Calabi-Yaus}
    Let $(M,g)$ be a Calabi-Yau manifold and denote by $\omega_g$ and $F_g \in \Omega^2(M;\su{T^\cn M})$ the Levi-Civita connection one-form and curvature two-form of $g$, respectively. Let $\psi^+$ and $\psi^-$ be orthogonal parallel spinors of positive resp. negative chirality. Then $(\omega_g, F_g\cdot \psi^\pm)$ are Dirac-Yang-Mills pairs on $M$ with coefficient bundle $E = \su{T^\cn M}$
\end{theorem}

\begin{remark}
If, in addition to being Yang-Mills, the connection $\omega$ is perturbation critical, we need not project onto chiral subspaces in Corollary \ref{thm: main thm multiplication by forms} and Theorem \ref{thm: construction with parallels}. Instead of two spaces of chiral solutions we then yield a single space of solutions of twice the dimension.
\end{remark}

\subsection{Explicit solutions on \texorpdfstring{$\s^4$}{S4} and \texorpdfstring{$\mathbb{T}^4$}{T4}}
We now wish to apply the preceding construction method to obtain some examples of Dirac-Yang-Mills pairs.
One can show (see \cite{BauFrGrKa}) that the only closed simply connected Riemannian $4$-manifolds $(M,g)$ admitting non-parallel twistor spinors are those of the form $(\s^4, g)$, where $g$ is a metric on $\s^4$ conformally equivalent to the standard round metric. For parallel spinors on $4$-manifolds this implies $M$ is either flat or a $K3$-surface \cite{BauFrGrKa}. In this section we take $M$ to be conformal to either $\s^4$ or to a standard $4$-torus, as these spaces yield many explicit examples. Here by standard $4$-torus we mean a flat torus $\mathbb{T}^4(\Gamma) = \rn^4/\Gamma$ for some lattice $\Gamma$ which we shall usually suppress in the notation.

By conformal invariance we shall work with the standard metrics on $\s^4$ and $\mathbb{T}^4$ in this section. For both $\mathbb{T}^4$ and $\s^4$ the twistor spinors are induced by twistor spinors on $\rn^4$. We note that $\Sigma\rn^4 \cong \rn^4\times \Sigma_4$ and solutions to the twistor equation \eqref{eq: twistor equation} on $\rn^4$ are given by \cite{BauFrGrKa}
\begin{align*}
    \phi(x) = \psi_0 + \frac{1}{4}x\cdot \psi_1
\end{align*}
for constant spinors $\psi_0,\psi_1 \in \Sigma_4$. 
This includes the constant spinors themselves which also trivially satisfy the required periodicity conditions to induce parallel spinors on $\mathbb{T}^4$ with its trivial spin structure. Thus there are two $2$-dimensional subspaces $V^+$ and $V^-$ of chiral twistor spinors corresponding to $\Sigma_4^+$ and $\Sigma_4^-$ respectively.

For $\s^4$ we use that $(\s^4 \setminus \mathrm{np},g_{\s^4})$ is conformally equivalent to $(\rn^4,g_{\rn^4})$, where $\mathrm{np} = (0,0,0,1)$. By the conformal invariance of \eqref{eq: twistor equation} the twistor spinors on $\rn^4$ yield solutions on $(\s^4 \setminus \mathrm{np},g_{\s^4})$ given by
\begin{align*}
    \psi(x) = \frac{\overline{{\psi_0} + x\cdot \psi_1}}{\sqrt{1+|x|^2}},
\end{align*}
where, as before, the bar denotes the image under the isomorphism $\Sigma_{g_{\rn^4}} \rn^4 \to \Sigma_{g_{\s^4}}(\s^4 \setminus \mathrm{np})$. These extend smoothly to all of $\s^4$ yielding an $8$-dimensional solution space. The $4$-dimensional spaces $V^\pm$ of chiral solutions are then explicitly given by taking $\psi_0 \in \Sigma_4^\pm$ and $\psi_1 \in \Sigma_4^\mp$. One then has the following:
\begin{proposition}[\cite{Hi1}]\label{prop: prod solutions sphere}
Let $M$ be $\s^4$ or $\mathbb{T}^4$ and $G \to P \to M$ be a $G$-principal fibre bundle over $M$ admitting a non-trivial Yang-Mills connection $\omega$. Then
$$
V^+_\omega = \{(\omega, F_\omega^+ \cdot \psi^+) \ | \ \psi^+ \in V^+\}
$$
and
$$
V^-_\omega = \{(\omega, F_\omega^- \cdot \psi^-) \ | \ \psi^- \in V^-\}
$$
are families of Dirac-Yang-Mills pairs with the coefficient bundle $E = \ad(P)$, at least one of which is nonempty.
\end{proposition}
\begin{remark}
Existence of self-dual connections on $\SU{N}$-principal bundles over certain tori $\mathbb{T}^4(\Gamma)$ is shown in \cite{Hooft}. Here the lattice is induced by an orthogonal basis and the lengths of the basis vectors are required to satisfy certain algebraic relations. Since the constructions there are explicit, these also yield explicit examples of Dirac-Yang-Mills pairs on tori. Taking any such connection $\omega$ and a constant positive chirality spinor $\psi_0$ the pair $(\omega, F_\omega \cdot \psi_0)$ is a Dirac-Yang-Mills pair on $\mathbb{T}^4$.
\end{remark}

\begin{remark}
If the connection $\omega$ in Proposition \ref{prop: prod solutions sphere} is neither self-dual nor anti-self-dual this construction yields both positive and negative chirality Dirac-Yang-Mills pairs $(\omega, \Psi)$. In the $\s^4$-case, by Corollary \ref{cor: instantons are perturbation minimal}, since $S_{g_{\s^4}} > 0$ (anti-)self-dual connections $\vartheta$ are perturbation minimal with $k_{\mathrm{pm}}(\ad(P)) = |\mathrm{ind}(\D_{\ad(P)}^+)| = \dim \ker (\D_{\vartheta})$. However $\dim \ker(\D_\omega) > \dim \ker (\D_{\vartheta})$ and so $\omega$ is an example of a non perturbation minimal connection. Such connections are shown to exist in \cite{MR1067574,MR1023811}. In the former, the $\SU{2}$-bundle is the trivial one, while on the latter non-(anti-)self-dual connections are constructed on a family of $\SU{2}$-bundles $P$ over $\s^4$ with $c_2(E) = k$ for any $k \neq 1$, where $E = \asb{P}{\cn^2}{\mathrm{st}}$. In either case Proposition \ref{prop: prod solutions sphere} gives two complex $4$-dimensional families of Dirac-Yang-Mills pairs.
\end{remark}

\begin{remark}
The (anti)-self dual Yang-Mills connections on $\SU{N}$, $\Sp{N}$ and $\O{N}$ bundles over $\s^4$ are well understood via the ADHM construction \cite{ADHM}. This method works by establishing a correspondence between certain sets of complex matrices and connections on vector bundles over $\s^4$. The self-duality equations then become a system of quadratic constraints for the matrix data. Taking $G$ to be one of $\SU{N}$, $\Sp{N}$ or $\O{N}$ this construction yields for each integer $k > 0$ a set of pairs $(E,\omega)$, where $E$ is a vector bundle with structure group $G$ (i.e. the frame bundle admits a reduction to a $G$-principal fibre bundle $P_G$), $c_2(E) = k$ and $\omega$ a self-dual connection on the frame bundle of $E$. anti-self-dual connections with $k < 0$ are obtained by reversing the orientation of the sphere.

Taking the specific case $G =\SU{N}$ we have $E = \asb{P_G}{\cn^N}{\mathrm{st}}$. Fixing $k \in \mathbb{Z}\setminus\{0\}$, in \cite{GodCor} the $k$ linearly independent solutions to the Dirac equation
\begin{align*}
    \D_\omega \Psi = 0,
\end{align*}
with $\Psi \in \Gamma(\Sigma^+ M \otimes E)$, are described explicitly in terms the matrix data going into the ADHM construction.
By Corollary \ref{cor: chiral spinor vanishing result} the associated pairs $(\omega, \Psi)$ also constitute an explicit family of Dirac-Yang-Mills pairs.
\end{remark}
If we instead consider the adjoint representation, by Lemma \ref{lem: index st vs index ad} we have 
$$
\mathrm{ind}(\D^+_{(\omega,\ad(P)}) = 2Nk
$$ 
while the complex dimension of the space of solutions given by Proposition \ref{prop: prod solutions sphere} is $4$. The construction thus yields all solutions if we take $N = 2$ and $k = \pm 1$.

\begin{example}
As an explicit example we apply our construction to the basic, or BPST, instanton on $\s^4$. This is the simplest example of the ADHM construction with $k = -1$ and also admits a description in terms of a quaternionic version of the Hopf-fibration.

Set $P = \s^7$ and let $\pi: \s^7 \to \s^4$ be the Hopf-fibration considered as an $\SU{2}$-principal bundle over $\s^4$. More precisely, we identify $\s^7$ with the unit sphere in $\mathbb{H}^2$ and $\s^4$ with the quaternionic projective line $\mathbb{H}\mathrm{P}^1$. Then $\pi:\s^7\subset \mathbb{H}^2 \to \mathbb{H}\mathrm{P}^1$ is given by the projection
\begin{equation*}
    \pi:(q_1,q_2) \mapsto [q_1\mathbin{:}q_2].
\end{equation*}
The fibres are isomorphic to $\s^3 \cong \SU{2}$ and the group action of $\SU{2} \cong \Sp{1}$ is given by quaternion multiplication. 

The orthogonal complement to the vertical distribution of $\pi$ with respect to the standard round metric on $\s^7$ defines a connection on $P$ in the sense of a horizontal distribution. The basic instanton $\vartheta$ is the connection one-form associated to this connection. Explicitly, we have
\begin{equation*}
    \vartheta_{q} = \mathrm{Im}(\bar{q}_1dq_1 + \bar{q}_2dq_2).
\end{equation*}
for $q = (q_1,q_2) \in \s^7$.
Set $U_0 = \s^4 \setminus \{[1\mathbin{:}0]\}$ and $U_\infty = \s^4 \setminus\{[0\mathbin{:}1]\}$ and define local charts $p_0:U_0 \to \s^7$ and $p_\infty:U_\infty \to \s^7$ by
\begin{equation*}
    p_0([q\mathbin{:}1]) \mapsto \frac{1}{\sqrt{1+|q|^2}}(q,1) \quad \text{and} \quad p_\infty([1\mathbin{:}q]) \mapsto \frac{1}{\sqrt{1+|q|^2}}(1,q).
\end{equation*}
The curvature of $\vartheta$ is locally given by
\begin{equation*}
    F_{\vartheta_0}:=p_0^*F_\vartheta = \frac{d\bar{q}\wedge dq}{(1+|q|^2)^2}
\end{equation*}
with a similar expression for $F_{\vartheta_\infty}$. A direct computation shows that $F_\vartheta$ is anti-self-dual with instanton number $k = -1$. From the previous discussion we then have that the $\ad(P)$-valued spinor fields locally given by
\begin{equation*}
    \Psi = F_{\vartheta_0} \cdot \frac{\overline{{\psi_0} + q\cdot \psi_1}}{\sqrt{1+|q|^2}} = {d\bar{q}\wedge dq\cdot {(1+|q|^2)^{-\frac{5}{2}}}\overline{(\psi_0 + q\cdot\psi_1 )}}
\end{equation*}
with $\psi_0 \in \Sigma^-_4$ and $\psi_1 \in \Sigma^+_4$ span the space of Dirac-Yang-Mills pairs $(\vartheta,\Psi)$. For a calculation of the eigenvalues and kernel of the basic instanton with the coefficient bundle $E = \asb{P}{\cn^2}{\mathrm{\rho}}$, where $\rho$ is an irreducible representation of $\SU{2}$, we refer to the work of Helga Baum in \cite{Bau}.
\end{example}

\section{Acknowledgements}
The author is very grateful to Volker Branding for his support and valuable feedback. The author also gratefully acknowledges the support of the Austrian Science Fund (FWF) through the project "The Standard Model as a Geometric Variational Problem"  (DOI: 10.55776/P36862).

\bibliography{bibliography.bib}

\bibliographystyle{amsplain}

\end{document}